\RequirePackage{luatex85}
\documentclass[leqno, 10pt]{amsart}
\usepackage{mathrsfs}
\usepackage{tikz-cd}
\usepackage{mathdots}
\usetikzlibrary{arrows,arrows.meta}
\usepackage[mathscr]{euscript}
\usepackage{algorithm}
\usepackage{algpseudocode}
\usepackage{hyperref}
\usepackage{amsmath}
\usepackage{amsthm}
\usepackage{amsopn}
\usepackage{amssymb}
\usepackage[all]{xy}
\usepackage{graphicx}
\usepackage{mathtools}

\allowdisplaybreaks[1]

\hypersetup{
    colorlinks,
    citecolor=black,
    filecolor=black,
    linkcolor=black,
    urlcolor=black
}
\usepackage{subcaption}

\makeatletter
\renewcommand\subsection{\leftskip 0pt\@startsection{subsection}{2}{\z@}%
                                     {-3.25ex\@plus -1ex \@minus -.2ex}%
                                     {1.5ex \@plus .2ex}%
                                     {\normalfont\normalsize\bfseries}}
\renewcommand\subsubsection{\@startsection{subsubsection}{3}{\z@}%
                                     {-3.25ex\@plus -1ex \@minus -.2ex}%
                                     {1.5ex \@plus .2ex}%
                                     {\normalfont\normalsize\bfseries\leftskip 3ex}}
\makeatother

\DeclareMathOperator{\Cell}{Cell}
\DeclareMathOperator{\DK}{DK}

\DeclareMathOperator{\Be}{Be}
\DeclareMathOperator{\SH}{\categ{SH}}
\DeclareMathOperator{\Sper}{Sper}

\DeclareMathOperator{\Sing}{Sing}

\DeclareMathOperator{\cell}{cell}

\renewcommand{\Sing}{\operatorname{Sing}}

\DeclareMathOperator{\CB}{C}
\DeclareMathOperator{\Tot}{Tot}

\newcommand{\mc}[1]{\mathcal{#1}}
\newcommand{\mcb}[1]{{\mathcal{#1}}}
\newcommand{\mb}[1]{\mathbb{#1}}
\newcommand{\mr}[1]{\mathrm{#1}}
\newcommand{\mbf}[1]{\mathbf{#1}}
\newcommand{\mit}[1]{\mathit{#1}}

\newcommand{\abs}[1]{\lvert #1 \rvert}

\newcommand{\bra}[1]{\langle #1 \rangle}

\newcommand{\ul}[1]{\underline{#1}}

\newcommand{\td}[1]{\widetilde{#1}}

\newcommand{\RR}{\mathbb{R}}
\newcommand{\ZZ}{\mathbb{Z}}
\newcommand{\CC}{\mathbb{C}}

\newcommand{\FF}{\mathbb{F}}
\newcommand{\GG}{\mathbb{G}}

\newcommand{\AF}{\mathbb{A}}
\newcommand{\PP}{\mathbb{P}}

\theoremstyle{definition}

 \newtheorem{thm}[equation]{Theorem}
 \newtheorem{cor}[equation]{Corollary}
 \newtheorem{lem}[equation]{Lemma}
 \newtheorem{prop}[equation]{Proposition}

 \newtheorem{rmk}[equation]{Remark}
\newtheorem*{thm*}{Theorem}
\newtheorem*{cor*}{Corollary}
\newtheorem*{lem*}{Lemma}
\newtheorem*{prop*}{Proposition}
\newtheorem*{defn*}{Definition}
\newtheorem*{ex*}{Example}
\newtheorem*{exs*}{Examples}
\newtheorem*{rmk*}{Remark}
\newtheorem*{claim*}{Claim}

\numberwithin{equation}{section}
\numberwithin{figure}{section}
\DeclareMathOperator{\Ext}{Ext}

\DeclareMathOperator{\Hom}{Hom}

\DeclareMathOperator{\Map}{Map}
\DeclareMathOperator*{\holim}{holim}

\DeclareMathOperator*{\colim}{colim}
\DeclareMathOperator*{\Spec}{Spec}

\RequirePackage{amssymb, latexsym}

\usepackage{enumerate, mdwlist}

\theoremstyle{definition}

\newtheorem{dfn}[equation]{Definition}

\newtheorem{exm}[equation]{Example}

\newtheorem{wrn}[equation]{Warning}

\newtheorem*{dfn*}{Definition}
\newtheorem*{axm*}{Axiom}
\newtheorem*{ntn*}{Notation}
\newtheorem*{exm*}{Example}
\newtheorem*{exr*}{Exercise}
\newtheorem*{int*}{Intuition}
\newtheorem*{qst*}{Question}
\newtheorem*{wrn*}{Warning}

\theoremstyle{plain}

\newtheorem*{cnj*}{Conjecture}

\DeclareMathOperator{\cofib}{cofib}

\DeclareMathOperator{\Fun}{Fun}
\DeclareMathOperator{\Ind}{Ind}

\DeclareMathOperator{\Res}{Res}

\newcommand{\cC}{\mathcal{C}}
\newcommand{\cD}{\mathcal{D}}

\newcommand{\cU}{\mathcal{U}}

\newcommand{\cX}{\mathcal{X}}

\newcommand{\cZ}{\mathcal{Z}}

\newcommand{\categ}[1]{{\textup{#1}}}

\newcommand{\CAlg}{\categ{CAlg}}

\newcommand{\Mod}{\categ{Mod}}

\newcommand{\Shv}{\categ{Shv}}
\newcommand{\Sp}{\categ{Sp}}
\renewcommand{\Pr}{\categ{Pr}}

\newcommand{\Sect}{\categ{Sect}}

\newcommand{\cocart}{\textit{cocart}}

\newcommand{\id}{\textit{id}}
\newcommand{\op}{\textit{op}}

\newcommand{\coloneq}{\mathrel{\mathop:}=}

\def\reflectedop#1#2{\mathop{\reflectbox{$#1{#2}$}}}

\def\noloc{{\mathpalette\reflectedop\colon}}
\def\reflectedop#1#2{\mathop{\reflectbox{$#1{#2}$}}}

\def\reflectedop#1#2{\mathop{\reflectbox{$#1{#2}$}}}

\newcommand*{\asterism}{\par\bigskip\noindent\hfill{\ooalign{\hfil\raise1ex\hbox{$\ast$}\hfil\cr$\ast$\kern-0.07ex$\ast$}}\hfill\null\par\bigskip}

\usepackage{tikz}
\usetikzlibrary{matrix,arrows,calc,decorations,decorations.pathmorphing,shapes}

\pgfdeclaredecoration{single line}{initial}{\state{initial}[width=\pgfdecoratedpathlength-1sp]{\pgfpathmoveto{\pgfpointorigin}}\state{final}{\pgfpathlineto{\pgfpointorigin}}}

\newcounter{sarrow}

\newcommand{\adjunct}[4]{{#1}\colon{#2}\ \begin{tikzpicture}[baseline] \draw[>=stealth,->] (0,1ex) -- (0.75,1ex); \draw[>=stealth,->] (0.75,0.25ex) -- (0,0.25ex); \end{tikzpicture}\ {#3}\noloc{#4}}

\renewcommand{\to}{\ \tikz[baseline]\draw[>=stealth,->](0,0.5ex)--(0.4,0.5ex);\ }

\newcommand{\righthookto}{\ \tikz[baseline]\draw[>=stealth,right hook->](0,0.5ex)--(0.4,0.5ex);\ }

\newcommand{\xto}[1]{\ \tikz[baseline]\draw[>=stealth,->,inner sep=2pt](0,0.5ex)--node[above]{\scriptsize $#1$}(0.5,0.5ex);\ }

\title{$C_2$-equivariant stable homotopy from real motivic stable homotopy}

\author{Mark Behrens}
\address{
Dept. of Mathematics \\
University of Notre Dame \\
Notre Dame, IN, U.S.A.
}
\author{Jay Shah}
\address{
Dept. of Mathematics \\
University of Notre Dame \\
Notre Dame, IN, U.S.A.
}

\begin{document}

\tikzcdset{arrow style=tikz, diagrams={>=stealth}}

\begin{abstract}
We give a method for computing the $C_2$-equivariant homotopy groups of the Betti realization of a $p$-complete cellular motivic spectrum over $\RR$ in terms of its motivic homotopy groups. More generally, we show that Betti realization presents the $C_2$-equivariant $p$-complete stable homotopy category as a localization of the $p$-complete cellular real motivic stable homotopy category.
\end{abstract}

\maketitle

\tableofcontents

\section{Introduction}

Let $\SH(K)$ denote the $\infty$-category of motivic spectra over a field $K$ \cite{MorelVoevodsky}, whose equivalences are given by the stable $\AF^1$-equivalences.  This $\infty$-category has a bigraded family of spheres
$$ S^{i,j} := S^{i-j} \wedge \GG_m^j $$
of topological degree $i$ and motivic weight $j$.  These lead to bigraded homotopy groups 
$$ \pi^K_{i,j}X := [S^{i,j}, X]_K. $$
A motivic spectrum is \emph{cellular} if it is built from the spheres $S^{i,j}$ using cofiber sequences and filtered homotopy colimits.  A map between cellular spectra is an stable $\AF^1$-equivalence if and only if it is a $\pi^K_{*,*}$-isomorphism \cite{DuggerIsaksencell}.  We shall let $\SH_{\cell}(K)$ denote the full subcategory of cellular spectra.

\subsection*{Complex and real Betti realization}

If $Z$ is a smooth scheme over $\CC$, then its $\CC$-points, or Betti realization,  
$$ \Be (Z) := Z(\CC) $$
form a topological space when endowed with the complex analytic topology.  The resulting Betti realization functor
$$ \Be: \SH(\CC) \to \Sp $$
(where $\Sp$ denotes the $\infty$-category of spectra)
is called \emph{Betti realization} \cite{MorelVoevodsky}.
Since $\Be(S^{i,j}) = S^i$, Betti realization induces a map
$$ \Be: \pi^{\CC}_{i,j} X \to \pi_{i}\Be({X}). $$ 

This map was well studied by Dugger and Isaksen \cite{DuggerIsaksenMASS} (at the prime $2$) and by Stahn \cite{Stahn} (at odd primes).  For a prime $p$, the $p$-complete motivic stable stems has an element
$$ \tau \in \pi^{\CC}_{0,-1} (S^{0,0})^{\wedge}_p. $$
The following result is a direct corollary of the results of Dugger-Isaksen and Stahn (here, $\widehat{\Be}_p(-)$ denotes $p$-completed Betti realization).

\begin{thm}[see Thm.~\ref{thm:ComplexSingFullyFaithful}]\label{thm:Cbetti}
Let $X \in \SH(\CC)$ be $p$-complete and cellular.  Then Betti realization induces an isomorphism of abelian groups
$$
\pi^{\CC}_{i,j}X[\tau^{-1}] \xto{\cong} \pi_i\widehat{\Be}_p(X), $$ 
and thus an equivalence of $\infty$-categories
$$ \widehat{\Be}_p: \SH_{\mr{cell}}({\CC})^{\wedge}_p[\tau^{-1}] \xto{\simeq} \Sp^{\wedge}_p. $$
\end{thm}

In the real case, there is a real Betti realization functor
$$ \Be_\RR: \SH(\RR) \to \Sp $$
which arises from associating to a smooth scheme $Z$ over $\RR$ its topological space of $\RR$-points $Z(\RR)$, endowed with the real analytic topology.  The inclusion
$$ \rho: \{\pm 1\} \righthookto \GG_m $$
gives an element $\rho \in \pi^{\RR}_{-1,-1}S^{0,0}$, which becomes an equivalence after real Betti realization.
Bachmann proved the following in \cite{Bachmann}.

\begin{thm}[see Thm.~\ref{thm:Bachmann}]\label{thm:Rbetti}
For all $X \in \SH(\RR)$, real Betti realization induces an isomorphism of abelian groups
$$ \Be_\RR: \pi^{\RR}_{i,j} X[\rho^{-1}] \xto{\cong} \pi_{i-j} \Be_\RR(X), $$
and moreover\footnote{Bachmann's methods do not rely upon cellularity hypotheses.} an equivalence of $\infty$-categories
$$ \SH({\RR})[\rho^{-1}] \xto{\simeq} \Sp. $$
\end{thm}

\subsection*{Statement of results}

The results discussed above
demonstrate that the homotopy groups of the complex and real Betti realizations of a cellular motivic spectrum can be obtained by localizing its motivic homotopy groups, and each of these Betti realization functors is a localization.

\emph{The purpose of this paper is to prove a similar result about the $C_2$-Betti realization functor}
$$ \Be^{C_2}: \SH(\RR) \to \Sp^{C_2}. $$
Here, $\Sp^{C_2}$ denotes the $\infty$-category of genuine $C_2$-spectra. This functor arises from associating to a smooth scheme $Z$ over $\RR$ the $C_2$-topological space $Z(\CC)$, with the $C_2$-action given by complex conjugation.

For $Y \in \Sp^{C_2}$, the $RO(C_2)$-graded equivariant homotopy groups are bigraded by setting 
$$ \pi^{C_2}_{i,j}Y := [S^{(i-j)+j\sigma}, Y]^{C_2}, $$
where $\sigma$ is the sign representation.
In $G$-equivariant homotopy theory, one takes the stable equivalences to be the $\pi^H_*$-isomorphisms, where $\pi_*^H$ denotes the $\mb{Z}$-graded $H$-equivariant homotopy groups, and $H$ ranges over the subgroups of $G$.  
However, in the case of $G = C_2$, a map in $\Sp^{C_2}$ is a stable equivalence if and only if it is a $\pi_{*,*}^{C_2}$-isomorphism (see the discussion following (\ref{eq:cofiber})).
The $C_2$-equivariant homotopy groups of $Y$ can be effectively analyzed from the homotopy pullback (isotropy separation square) \cite{GreenleesMay}
\begin{equation}\label{eq:Tate}
\xymatrix{
Y \ar[r] \ar[d] & Y^{\Phi} \ar[d] \\
Y^h \ar[r] & Y^t
}
\end{equation}
where\footnote{The terminology here comes from the fact that the fixed points $Y^{hC_2}$, $Y^{\Phi C_2}$, and $Y^{tC_2}$ are the homotopy fixed points, geometric fixed points, and Tate spectrum of $Y$, respectively.}
\begin{align*}
Y^h & := F((EC_2)_+, Y), & \text{(homotopy completion)}, \\
Y^\Phi & := Y \wedge \td{EC_2}, & \text{(geometric localization)}, \\
Y^t & := (Y^h)^\Phi, & \text{(equivariant Tate spectrum)}.
\end{align*}
We let $\Sp^{hC_2}$ denote the full subcategory of $\Sp^{C_2}$ consisting of homotopically complete spectra, and let $\Sp^{\Phi C_2}$ denote the full subcategory consisting of geometrically local spectra.  The $C_2$-geometric fixed points functor gives an equivalence of $\infty$-categories $\Sp^{\Phi C_2} \simeq \Sp$.

Bachmann's theorem (Theorem~\ref{thm:Rbetti}) effectively describes the homotopy theory of the geometric localization of $C_2$-Betti realization $\Be^{C_2}(-)^{\Phi}$.  This is because
\begin{enumerate}
\item for all $X \in \SH(\RR)$, we have
$$ \Be_\RR(X) = \Be^{C_2}(X)^{\Phi C_2}, $$
\item geometric localization is given by inverting $a := \Be^{C_2}(\rho) \in \pi^{C_2}_{-1,-1} S^{0,0}$:
\begin{equation}\label{eq:Phi}
 Y^\Phi \simeq Y[a^{-1}].
 \end{equation}
\end{enumerate}

Thus, Bachmann's theorem (Theorem \ref{thm:Rbetti}) can be restated in the following way.

\begin{thm}[see Thm.~\ref{thm:Bachmann2}]\label{thm:C2betti}
For all $X \in \SH(\RR)$, $C_2$-Betti realization induces an isomorphism
$$ \pi^{\RR}_{*,*}X[\rho^{-1}] \xto{\cong} \pi^{C_2}_{*,*}\Be^{C_2}(X)^\Phi, $$
and an equivalence
$$ \Be^{C_2}: \SH(\RR)[\rho^{-1}] \xto{\simeq} \Sp^{\Phi C_2}. $$ 
\end{thm}

We are thus left to describe the homotopy theory of the homotopy completion of the $C_2$-Betti realization.

We first note that a map 
$$ f: Y_1 \to Y_2 $$
in $\Sp^{hC_2}$ is an equivalence if and only if the underlying map
$$ f^e: Y_1^e \to Y_2^e $$ 
of spectra is a non-equivariant equivalence. We therefore first study $\Be^{C_2}(-)^e$.
Consider the diagram of adjoint functors
\begin{equation}
\xymatrix@R+1em@C+1em{
\SH({\RR}) \ar@<.5ex>[r]^{\Be^{C_2}} \ar@<-.5ex>[d]_{\zeta^*}   & \Sp^{C_2} \ar@<.5ex>[l]^{\Sing^{C_2}} \ar@<-.5ex>[d]_{\Res_e^{C_2}} \\
\SH({\CC}) \ar@<.5ex>[r]^{\Be} \ar@<-.5ex>[u]_{\zeta_*} & \Sp \ar@<-.5ex>[u]_{\Ind_e^{C_2}}
\ar@<.5ex>[l]^{\Sing} 
}
\end{equation}
where $(\zeta^*, \zeta_*)$ are the base change functors associated to the morphism
$$ \zeta: \Spec(\CC) \to \Spec(\RR) $$
and $(\Res_e^{C_2},\Ind_e^{C_2})$ are the change of group functors associated to the inclusion
$$ e \righthookto C_2. $$
We will prove the following theorem, which has also been independently obtained by Isaksen-Kong-Wang-Xu.

\begin{thm}[see Cor.~\ref{cor:EtaleMonoidalBarrBeck}]\label{thm:Cmod}
Under the equivalence
$$ \zeta_*\zeta^* S^0 \simeq \Sigma^\infty_+ \Spec(\CC) $$
the adjunction $(\zeta^*,\zeta_*)$ induces an equivalence
$$ \SH(\CC) \simeq \Mod_{\SH(\RR)}(\Sigma^\infty_+\Spec(\CC)). $$
\end{thm}

Let $C\rho \in \SH(\RR)$ denote the cofiber of $\rho \in \pi^{\RR}_{-1,-1}S^{0,0}$.  Since $\zeta^*(\rho)$ is null in $\pi^{\CC}_{-1,-1}S^{0,0}$, there is a map
\begin{equation}\label{eq:Crho}
 C(\rho) \to \Sigma^\infty_+ \Spec(\CC)
 \end{equation}
which we show is a $\pi^{\RR}_{*,*}$-isomorphism after $p$-completion (Prop.~\ref{prop:cofiberRho}).  The real motivic spectrum $\Sigma^\infty_+ \Spec(\CC)$ is not cellular (Rmk.~\ref{rmk:cofiberRho}), so $C\rho$ may be regarded as its $p$-complete cellular approximation.  We deduce the following (which was also independently observed by Isaksen-Kong-Wang-Xu):

\begin{cor}[see Cor.~\ref{cor:CellularBarrBeck}]\label{cor:Cmod}
The adjuction $(\zeta^*, \zeta_*)$ and equivalence (\ref{eq:Crho}) induces an equivalence
$$ \SH_{\cell}(\CC)^{\wedge}_p \simeq \Mod_{\SH_{\mr{cell}}({\RR})^{\wedge}_p}(C(\rho)). $$
In particular, for $X \in \SH_{\mr{cell}}({\RR})^{\wedge}_p$ there is an isomorphism
$$ \pi^{\CC}_{*,*}(\zeta^* X) \cong \pi^{\RR}_{*,*}(X \wedge C\rho). $$
\end{cor} 

Combining Corollary~\ref{cor:Cmod} with Theorem~\ref{thm:Cbetti}, we deduce that for $X \in \SH_{\mr{cell}}(\RR)^{\wedge}_p$,
$$ \pi_{i}(\widehat\Be{}_p^{C_2}({X})^e) \cong \pi^{\RR}_{i,j} (X \wedge C\rho[\tau^{-1}]). $$
In particular, $\tau$ exists as a self-map
$$ \tau: \Sigma^{0,-1} C(\rho)^{\wedge}_p \to C(\rho)^\wedge_p. $$

Let $C(\rho^i)$ denote the cofiber of $\rho^i \in \pi^{\RR}_{-i,-i}S^{0,0}$.  We will prove:

\begin{thm}[see Thm.~\ref{thm:tauselfmap} and Prop.~\ref{prop:tauodd}]\label{thm:selfmap}
For each $i \ge 1$, there exists a $j$ so that $C(\rho^i)^{\wedge}_p$ has a $\tau^j$ self map
$$ \tau^j: \Sigma^{0,-j} C(\rho^i)^{\wedge}_p \to C(\rho^i)^{\wedge}_p. $$ 
\end{thm}

Our proof of the existence of these $\tau$-self maps at the prime $2$ relies on first proving the existence of their  $C_2$-Betti realizations,  and then using a theorem of Dugger-Isaksen to lift the self maps to the real motivic category \cite{DuggerIsaksenC2}.  Because this approach involves some analysis of the $C_2$-equivariant stable stems, it may be of independent interest. 

We shall let $C(\rho^i)^{\wedge}_p[\tau^{-1}]$ denote the telescope of this $\tau^j$-self map.  Define, for $X \in \SH(\RR)^{\wedge}_p$:\footnote{For $p$ odd, it turns out that independent of $i$, one can take $j = 2$ in Thm.~\ref{thm:selfmap} (see Prop.~\ref{prop:tauodd}).  Consequently, $X^{\wedge}_\rho$ has a $\tau^2$-self map, and the spectrum $X^{\wedge}_\rho[\tau^{-1}]$ can be simply taken to be the telescope of this $\tau^2$-self map on $X^\wedge_\rho$.}
$$ X^{\wedge}_\rho[\tau^{-1}] := \holim_i X \wedge C(\rho^i)[\tau^{-1}]. $$

Our main theorem is the following:

\begin{thm}[see Cor.~\ref{cor:BorelFullyFaithful} and Thm.~\ref{thm:main}]\label{thm:mainintro}
For $X \in \SH_{\mr{cell}}({\RR})^{\wedge}_p$, the $p$-completed $C_2$-Betti realization functor $\widehat\Be{}_p^{C_2}$ induces an isomorphism
$$ \pi^{\RR}_{*,*} X^{\wedge}_\rho[\tau^{-1}] \xto{\cong} \pi^{C_2}_{*,*} \widehat\Be{}_p^{C_2} (X)^h $$
and the right adjoint 
$$ \Cell \Sing^{C_2}: (\Sp^{hC_2})^{\wedge}_p \to \SH_{\mr{cell}}({\RR})^{\wedge}_p $$
of $p$-complete, homotopy complete $C_2$-Betti realization
is fully faithful.
\end{thm} 

Thus, Theorems~\ref{thm:C2betti} and \ref{thm:mainintro} combine to express the $RO(C_2)$-graded equivariant homotopy groups of $\widehat\Be{}_p^{C_2}(X)^\Phi$ and $\widehat\Be{}_p^{C_2}({X})^h$ in terms of the real motivic homotopy groups of $X$.  By the isotropy separation square (\ref{eq:Tate}), we just need to be able to compute $\pi^{C_2}_{*,*}\widehat\Be{}_p^{C_2}(X)^t$ (and the maps on homotopy groups) to recover $\pi^{C_2}_{*,*} \widehat\Be{}_p^{C_2} (X)$, but this is easily accomplished by combining Theorem~\ref{thm:main} with (\ref{eq:Tate}) to deduce (for $X$ cellular and $p$-complete) an isomorphism\footnote{For $p$ odd, the situation is much simpler, as this Tate spectrum is contractible since $2 = \abs{C_2}$ is invertible.}
$$ \pi^{C_2}_{*,*}\widehat\Be{}_p^{C_2}(X)^t \cong \pi^{\RR}_{*,*}X^\wedge_\rho[\tau^{-1}][\rho^{-1}]. $$

Finally, we will show that the isotropy separation square (\ref{eq:Tate}) implies that Theorems~\ref{thm:C2betti} and \ref{thm:mainintro} combine to show that $p$-complete $C_2$-equivariant stable homotopy is a localization of real motivic cellular stable homotopy.

\begin{thm}[see Thm.~\ref{thm:C2-SingFullyFaithful}]
The right adjoint to $p$-complete cellular $C_2$-Betti realization
$$ \Cell \Sing^{C_2}: (\Sp^{C_2})^{\wedge}_p \to \SH_{\cell}(\RR)^{\wedge}_p $$
is fully faithful.
\end{thm}

We will apply our techniques to compute $\pi^{C_2}_{*,*} \widehat\Be{}_2^{C_2} X$ from $\pi^\RR_{*,*}X$, for $X$ equal to: 
\begin{enumerate}
\item $(H\FF_2)_\RR$, the real motivic mod $2$ Eilenberg-MacLane spectrum, with 
$$ \widehat\Be{}_2^{C_2} (H\FF_2)_\RR \simeq H\ul{\FF_2}, $$
the $C_2$-equivariant Eilenberg-MacLane spectrum associated to the constant Mackey functor $\ul{\FF_2}$,\vspace{10pt}
\item $(H\ZZ^{\wedge}_2)_\RR$, the real motivic $2$-adic Eilenberg-MacLane spectrum, with 
$$ \widehat\Be{}_2^{C_2} (H\ZZ^\wedge_2)_\RR \simeq H\ul{\ZZ^\wedge_2}, $$
the $C_2$-equivariant Eilenberg-MacLane spectrum associated to the constant Mackey functor $\ul{\ZZ^\wedge_2}$, and \vspace{10pt}
\item $kgl^\wedge_2$, the $2$-complete effective cover of the real motivic $K$-theory spectrum $KGL$, with
$$ \widehat\Be{}_2^{C_2} kgl^\wedge_2 \simeq kR^{\wedge}_2, $$
the $2$-complete connective Real $K$-theory spectrum.
\end{enumerate}
In the case of $(H\FF_2)_\RR$, the homotopy groups of the $C_2$-Betti realization differ from the motivic homotopy groups of the original spectrum through the addition of a notorious ``negative cone'' (see, for example, \cite[Fig.~1]{DuggerIsaksenC2}).  From the perspective of the mod $2$-Adams spectral sequence, the presence of this ``negative cone'' makes the equivariant homotopy of the Betti realizations of the other examples similarly more complicated than the motivic homotopy of the original spectra.  Our theory organically predicts the presence of the negative cone through a mechanism of local duality such as that studied in \cite{BHV}, and thus gives a more direct route to these equivariant computations by starting with the simpler motivic analogs. This connection with local duality deserves further study.

\subsection*{Relationship to the work of Heller-Ormsby}

Heller and Ormsby \cite{HellerOrmsby}, \cite{HellerOrmsby2} also study the relationship between real motivic and $C_2$-equivariant spectra (and their results extend to other real closed fields), but their analysis centers around the adjoint pair
$$ \adjunct{c_\RR^*}{\Sp^{C_2}}{\SH(\RR)}{ (c_\RR)_*}$$
where $c_\RR^*$ is the equivariant generalization of the constant functor (Def.~\ref{def:HellerOrmsby}).

Namely, Heller and Ormsby show that $\Sp^{C_2}$ is a \emph{colocalization} of $\SH(\RR)$ by showing that $c_\RR^*$ is fully faithful.  Their results allow them to compute, for $X \in \Sp^{C_2}$, integer graded motivic homotopy groups of $c^*_\RR X$ in terms of the integer graded equivariant homotopy groups of $X$.

Our results, by contrast, show that $C_2$-Betti realization exhibits $(\Sp^{C_2})^\wedge_p$ as a \emph{localization} of $\SH_{\cell}(\RR)^{\wedge}_p$, and this allows us to compute, for $X \in \SH_{\cell}(\RR)^{\wedge}_p$, the equivariant $RO(C_2)$-graded homotopy groups of $\widehat\Be{}_p^{C_2}(X)$ in terms of the bigraded motivic homotopy groups of $X$.  Nevertheless, we use the functor $c^*_\RR$ to prove our localization theorem.

\subsection*{Organization of the paper}

The first four sections of this paper are formal.
In Section~\ref{PresentableLimitFacts}, we recall some facts concerning limits of presentable $\infty$-categories.

In Section~\ref{sec:localization-of-symmetric-monoidal-infty-categories-with-respect-to-a-commutative-algebra}, we study Bousfield localizations of symmetric monoidal $\infty$-categories, their relation to completion, and discuss the interaction of these localizations with a monoidal Barr-Beck theorem of Mathew-Naumann-Noel \cite{MATHEW2017994}.

In Section~\ref{sec:cellularization}, we summarize some facts regarding cellularization in the $\infty$-categorical context, and the interaction of cellularization with localization and symmetric monoidal structures.

In Section~\ref{sec:recollements}, we recall the notion of a \emph{recollement} of $\infty$-categories, which is a formalism for decomposing an $\infty$-category using two complementary localizations.  We show that to prove an adjunction between two recollements is a localization, it suffices to check fully faithfullness on the constituents of the recollements.

In Section~\ref{sec:background}, we turn to the case of interest and recall some facts about motivic and equivariant homotopy theory that we will need later.

In Section~\ref{sec:tau}, we show that James periodicity in the $2$-primary equivariant stable stems results from the existence of \emph{$u$-self maps} on $C(a^i)^{\wedge}_2$, where $a$ is the Euler class of the sign representation.  We then use an isomorphism theorem of Dugger-Isaksen \cite{DuggerIsaksenC2} to lift these $u$-self maps to \emph{$\tau$-self maps} on $C(\rho^i)^\wedge_2$.
For an odd prime $p$, we explain how the work of Stahn \cite{Stahn} implies that every $(p,\rho)$-complete $\RR$-motivic spectrum has a $\tau^2$-self map.

Section~\ref{sec:the-equivariant-motivic-situation} contains all of our main theorems, and their proofs, concerning the localizations induced by Betti realization.

Section~\ref{sec:examples} contains examples, where we take various real motivic spectra, and use our theory to compute the $2$-primary $RO(C_2)$-graded $C_2$-equivariant homotopy groups of their Betti realizations from their $2$-primary motivic homotopy groups.  We also explain how to do these kinds of computations at an odd prime, where the story is much simpler.

\subsection*{Acknowledgments}  The authors would like to thank Dan Dugger, Elden Elmanto, Bert Guillou, Jeremiah Heller, Dan Isaksen, Kyle Ormsby, James Quigley, Guozhen Wang, and Zhouli Xu for helpful conversations related to the subject matter of this paper.  The authors would also like to thank the referee for their very helpful corrections and comments.  The first author was supported by NSF grant DMS-1611786, and the second author was supported by NSF grant DMS-1547292.

\section{Limits of presentable \texorpdfstring{$\infty$}{infinity}-categories}\label{PresentableLimitFacts} 

We collect some necessary facts about limits in the $\infty$-category $\Pr^L$ of presentable $\infty$-categories. 

Suppose $\mcb{C}_{\bullet}: J \to \Pr^L$ is a diagram and let 
$$\mcb{X} = \int \mcb{C}_{\bullet} \to J$$ 
be the presentable fibration \cite[Dfn. 5.5.3.2]{HTT} classified by $\mcb{C}_{\bullet}$. By \cite[Prop. 5.5.3.13]{HTT} and \cite[Cor. 3.3.3.2]{HTT}, we have an equivalence
\[ \mcb{C} \coloneq \lim \mcb{C}_{\bullet} \simeq \Sect(\mcb{X}) \coloneq \Fun^{\cocart}_{/J}(J, \mcb{X}) \]
between the limit $\mcb{C}$ of $\mcb{C}_{\bullet}$ and the $\infty$-category of cocartesian sections of $\mcb{X}$. Let $\mcb{D}$ be another presentable $\infty$-category and suppose that we have an extension 
$$\overline{\mcb{C}}_{\bullet} : J^{\lhd} \to \Pr^L$$ 
with the cone point sent to $\mcb{D}$. Then we have an induced adjunction $$\adjunct{F}{\mcb{D}}{\mcb{C}}{R}.$$

Let 
$$\overline{\mcb{X}} \to J^{\lhd}$$ 
be the presentable fibration classified by $\overline{\mcb{C}}_{\bullet}$. In terms of the description of $C$ as $\Sect(\mcb{X})$, we may describe $F$ and $R$ more explicitly as follows:
\begin{enumerate}
    \item The functor 
    $$F: \mcb{D} \simeq \lim \overline{\mcb{C}}_{\bullet} \to \mcb{C} \simeq \lim \mcb{C}_{\bullet}$$ 
    is given by the contravariant functoriality of limits for the inclusion $J \subset J^{\lhd}$. Thus, under the equivalences $\Sect(\overline{\mcb{X}}) \simeq \mcb{D}$ and $\Sect(\mcb{X}) \simeq \mcb{C}$, the functor $F: \mcb{D} \to \mcb{C}$ corresponds to the functor 
    $$F: \Sect(\overline{\mcb{X}}) \to \Sect(\mcb{X})$$ 
    given by restriction of cocartesian sections. In particular, an object $x \in \mcb{D}$ corresponds to the cocartesian section 
    $$\overline{\sigma}: J^{\lhd} \to \overline{\mcb{X}}$$ 
    determined up to contractible choice by $\overline{\sigma}(v) = x$ for $v$ the cone point, and then $F(x) = \overline{\sigma}|_{J}$.
    \vspace{10pt}

    \item Let 
    $$p: \mcb{X} \subset \overline{\mcb{X}} \to \overline{\mcb{X}}_{v} \simeq \mcb{D}$$ 
    be the cartesian pushforward to the fiber over the initial object $v \in J^{\lhd}$. Then for any object $\sigma \in \mcb{C}$ viewed as a cocartesian section of $\mcb{X}$ and $x \in \mcb{D}$, we have the sequence of equivalences
    \begin{align*} \Map_{\mcb{D}}(x, \lim p \sigma) & \simeq \lim \Map_{\mcb{D}}(x, p \sigma(-) ) \\
                                            & \simeq \lim \Map_{\mcb{C}_{\bullet}}(F_{\bullet} x, \sigma(-)) \\
                                            & \simeq \Map_{\mcb{C}}(Fx, \sigma),
    \end{align*}
    so there is an equivalence $R(\sigma) \simeq \lim p \sigma$.
\end{enumerate}

\section{Localization of symmetric monoidal \texorpdfstring{$\infty$}{infinity}-categories with respect to a commutative algebra}\label{sec:localization-of-symmetric-monoidal-infty-categories-with-respect-to-a-commutative-algebra}

Let $\mcb{C}$, $\mcb{D}$ be presentable stable symmetric monoidal $\infty$-categories, where we by default assume that the tensor product commutes with colimits separately in each variable.  

\subsection*{Adjunctions and limits}

We say that an adjunction 
$$\adjunct{F}{\mcb{C}}{\mcb{D}}{R}$$ 
is \emph{monoidal} if $F$ is (strong) symmetric monoidal, in which case $R$ is lax monoidal.

Given a diagram of commutative algebras in $\mcb{C}$
$$p: L \to \CAlg(\mcb{C}),$$ 
we have a canonical monoidal adjunction in $\Pr^L$
\begin{equation}\label{eq:LimitComparisonMap}
 \adjunct{\phi}{\Mod_{\mcb{C}}(\lim_L p)}{\lim_L \Mod_{\mcb{C}}(p(-))}{\psi}. 
 \end{equation}
Let $R =\lim_L p$. For $X \in \Mod_{\mcb{C}}(R)$, the unit map $\eta: X \to \psi \phi X$ may be identified with the canonical map 
$$X \to \lim_{i \in L} X \otimes_{R} p(i)$$ 
in view of the material in Section~\ref{PresentableLimitFacts}.

Moreover, for any functor $f: K \to L$, by functoriality of limits we have a commutative diagram in $\Pr^L$
\[ \begin{tikzcd}[row sep=4ex, column sep=4ex, text height=1.5ex, text depth=0.25ex]
\Mod_{\mcb{C}}(\lim_L p) \ar{r}{\phi} \ar{d} & \lim_L \Mod_{\mcb{C}}(p(-)) \ar{d} \\
\Mod_{\mcb{C}}(\lim_K p f) \ar{r}{\phi} & \lim_K \Mod_{\mcb{C}}(p f(-) ).
\end{tikzcd} \]

\subsection*{Bousfield localization}

Recall \cite[Dfn.~5.2.7.2]{HTT} that a \emph{localization} of an $\infty$-category $\mc{X}$ is an adjunction
$$ \adjunct{L}{\mc{X}}{\mc{X}_0}{R}$$
where the right adjoint $R$ is fully faithful.  The left adjoint $L$ is the localization functor.

When $\mc{X} = \mc{C}$ is our presentable stable symmetric monoidal $\infty$-category, we will be concerned with the special case of \emph{Bousfield localization} with respect to an object $E \in \mcb{C}$. We briefly recall this notion to fix terminology.

A map 
$$X \to Y$$ 
in $\mcb{C}$ is an \emph{$E$-equivalence} if 
$$E \otimes X \to E \otimes Y$$ 
is an equivalence. An object $X \in \mcb{C}$ is \emph{$E$-null} if 
$$X \otimes E \simeq 0.$$ 
An object $X \in \mcb{C}$ is \emph{$E$-local} if for every $E$-equivalence 
$$f: Y \to Z$$ 
the map 
$$f^{\ast}: \Hom_{\mcb{C}}(Z,X) \to \Hom_{\mcb{C}}(Y,X)$$ 
is an equivalence; i.e. for every $E$-null object $W$, 
$$ \Hom_{\mcb{C}}(W,X) \simeq 0. $$
Let $\mcb{C}_E \subseteq \mcb{C} $ denote the full subcategory consisting of the $E$-local objects. Then $\mcb{C}_E$ is again a presentable stable $\infty$-category and we have the localization adjunction
\[ \adjunct{L_E}{\mcb{C}}{\mcb{C}_E}{i_E}. \]
With the tensor product on $\mcb{C}_E$ defined by $L_E(- \otimes -)$, $\mcb{C}_E$ is a symmetric monoidal $\infty$-category and $L_E \dashv i_E$ is a monoidal adjunction.

\begin{exm}\label{ex:Cx} Suppose $E = C(x)$ is the cofiber of a map
$$x: I \to 1$$ 
for $1 \in \mcb{C}$ the unit. Then we also write $\mcb{C}^{\wedge}_x$ for $\mcb{C}_E$ and call this $\infty$-category the \emph{$x$-completion} of $\mcb{C}$.
\end{exm}

\subsection*{Derived completion}

If we further suppose that $E$ is a dualizable $E_{\infty}$-algebra $A \in \CAlg(\mcb{C})$, then Bousfield localization can be computed as the $A$-completion.  Specifically, we have the following.
\begin{enumerate}
    \item Let $\CB^{\bullet}(A)$ be the Amitsur complex on $A$ (\cite[Constr. 2.7]{MATHEW2017994}). By \cite[Prop. 2.21]{MATHEW2017994}, for any $X \in \mcb{C}$ we have an equivalence
    \[ L_A(X) \simeq \Tot (X \otimes \CB^{\bullet}(A)) \simeq \lim_{n \in \Delta} (X \otimes A^{\otimes n+1}).  \]
    \item By \cite[Thm. 2.30]{MATHEW2017994}\footnote{Note that even though the hypotheses of \cite[2.26]{MATHEW2017994} are otherwise in effect in that section of the paper, the proof of \cite[Thm. 2.30]{MATHEW2017994} only uses that $A \in \CAlg(\mcb{C})$ is dualizable.}, this equivalence of objects promotes to an equivalence of symmetric monoidal $\infty$-categories
    \[ \mcb{C}_A \simeq \Tot \Mod_{\mcb{C}}(\CB^{\bullet}(A)) \simeq \lim_{n \in \Delta} \Mod_{\mcb{C}}(A^{\otimes n+1}). \]
\end{enumerate}

Let $I$ denote the fiber of the unit
$$ I \xto{\iota} 1 \to A $$ 
and define 
$$ C(\iota^n) := \cofib(\iota^n: I^{\otimes n} \to 1).$$ 
Then there is an equivalence \cite[Prop. 2.14]{MATHEW2017994}
$$C(\iota^{n+1}) \simeq \Tot_{n} (\CB^{\bullet}(A)).$$
Note that because the cosimplicial object $\CB^{\bullet}(A)$ in $\mcb{C}$ canonically lifts to a cosimplicial object in $\CAlg(\mcb{C})$, the cofiber $C(\iota^n)$ obtains the structure of an $E_{\infty}$-algebra as a limit and the maps 
$$ C(\iota^{n+1}) \to C(\iota^n)$$ 
are maps of $E_{\infty}$-algebras. 

\subsection*{The completion tower}

For our dualizable $E_\infty$-algebra $A$, we wish to re-express the above descent description of $\mc{C}_A$ in terms of an inverse limit over the $\infty$-categories $\Mod_{\mcb{C}}(C(\iota^n))$.

By \ref{eq:LimitComparisonMap}, 
for all $n$ we have canonical monoidal adjunctions
\[ \adjunct{\phi_n}{\Mod_{\mcb{C}}(C(\iota^n))}{\Tot_{n-1} \Mod_{\mcb{C}}(\CB^{\bullet}(A))}{\psi_n} \]
where the left adjoints $\phi_n$ are compatible with restriction along $\Delta_{\leq n} \subset \Delta_{\leq m}$. Passage to the limit then yields the monoidal adjunction
\[ \adjunct{\phi_{\infty}}{\lim \Mod_{\mcb{C}}(C(\iota^n)) }{\Tot \Mod_{\mcb{C}}( \CB^{\bullet}(A)) } {\psi_{\infty}}, \]
where $\phi_{\infty} \{X_n \} = \{ \phi_n X_n \}$. By the universal property of the limit, and using that $C(\iota^n)$-modules are $A$-local, we also have the monoidal adjunctions 
\begin{gather*}
 \adjunct{\phi}{\mcb{C}_A}{\lim \Mod_{\mcb{C}}(C(\iota^n))}{\psi}, \\ \adjunct{\phi'}{\mcb{C}_A}{\Tot \Mod_{\mcb{C}}( \CB^{\bullet}(A)) }{\psi'}
 \end{gather*}
where the second adjunction is the adjoint equivalence of \cite[Thm. 2.30]{MATHEW2017994}. These adjunctions fit into a commutative diagram
\[ \begin{tikzcd}[row sep=4ex, column sep=4ex, text height=1.5ex, text depth=0.25ex]
\mcb{C}_A  \ar{r}[swap]{\phi} \ar[bend left=15]{rr}{\phi'} & \lim \Mod_{\mcb{C}}(C(\iota^n)) \ar{r}[swap]{\phi_{\infty}}  & \Tot \Mod_{\mcb{C}}( \CB^{\bullet}(A)).
\end{tikzcd} \]

\begin{prop} \label{prop:DoldKan} $\phi \dashv \psi$ and $\phi_{\infty} \dashv \psi_{\infty}$ are adjoint equivalences of symmetric monoidal $\infty$-categories.
\end{prop}
\begin{proof} It suffices to prove the first statement. We need to show that
\begin{enumerate} \item The unit $\id \to \psi \phi$ is an equivalence.
\item $\psi$ is conservative.
\end{enumerate}

For (1), given any $A$-local object $X$ the unit map
 \[ X \to \psi \phi X \simeq \lim X \otimes C(\iota^n) \simeq \Tot X \otimes \CB^{\bullet}(A) \]
is already known to be an equivalence. For (2), we first note that $\phi_n$ is fully faithful, i.e., the unit map $\id \to \psi_n \phi_n$ is an equivalence. Indeed, for \emph{any} finite $\infty$-category $K$ and functor $p: K \to \CAlg(\mcb{C})$, if $R = \lim_K p$ and $X \in \Mod_C(R)$, then there is an equivalence 
$$ X \simeq \lim_{K} X \otimes_R p(-).$$ 
Now suppose that $\{ X_n \}$ is a object in $\lim \Mod_{\mcb{C}}(C(\iota^n))$ such that 
$$ \psi \{X_n \} = \lim X_n \simeq 0. $$
Note that since $\phi_{\infty} \{X_n \} = \{ \phi_n X_n \}$, for the cosimplicial object $\phi_{\bullet} X_{\bullet}$ we have that $\Tot_n (\phi_{\bullet} X_{\bullet}) \simeq \psi_n \phi_n X_n \simeq X_n$, so
\[ \psi'(\phi_{\infty} \{X_n \}) = \Tot \phi_{\bullet} X_{\bullet} \simeq \lim_n \Tot_n (\phi_{\bullet} X_{\bullet}) \simeq \lim_n X_n \simeq 0.    \]
Therefore, because $\psi'$ is an equivalence, $\phi_{\infty} \{X_n \} \simeq 0$. This means that for all $n$, $\phi_n X_n \simeq 0$, so $X_n \simeq \psi_n \phi_n X_n \simeq 0$ and $\{X_n\} \simeq 0$.
\end{proof}

\begin{rmk} We have a commutative diagram of right adjoints
\[ \begin{tikzcd}[row sep=4ex, column sep=4ex, text height=1.5ex, text depth=0.25ex]
\lim \Mod_{\mcb{C}}(C(\iota^n)) \ar{d} & \Tot \Mod_\cC( \CB^{\bullet}(A)) \ar{d} \ar{l}{\psi_{\infty}} \\
\Fun(\mb{Z}_{\ge 0}^{\op},\mcb{C}) & \Fun(\Delta, \mcb{C}) \ar{l}{\DK}.
\end{tikzcd} \]
where $\DK$ is the functor that sends a cosimplicial object to its tower of partial totalizations. $\DK$ implements the equivalence of the $\infty$-categorical Dold-Kan correspondence (\cite[Thm. 1.2.4.1]{HA}). We may thus interpret Prop. \ref{prop:DoldKan} as a monoidal refinement of the Dold-Kan correspondence, with $\phi_{\infty}$ providing an explicit inverse.
\end{rmk}

We also record a useful corollary of the proof of Prop.~\ref{prop:DoldKan}. This result is a companion to the fact that 
$$- \otimes A: \mcb{C}_A \to \Mod_{\mcb{C}}(A)$$ 
is conservative.

\begin{lem} \label{lem:BaseChangeConservative} For every $n$, the base-change functor 
$$\Mod_{\mcb{C}}(C(\iota^n)) \to \Mod_{\mcb{C}}(C(\iota)) = \Mod_{\mcb{C}}(A)$$ 
is conservative. 
\end{lem}

\begin{proof} We showed that the functor $\phi_n$ is fully faithful, and the restriction functor
\[ \Tot_{n-1} \Mod_{\mcb{C}}(\CB^{\bullet}(A)) \to \Mod_{\mcb{C}}(A) \]
is clearly conservative.
\end{proof}

\subsection*{The monoidal Barr-Beck theorem}

Throughout, let $$\adjunct{F}{\mcb{C}}{\mcb{D}}{R}$$ 
be a monoidal adjunction 
between our presentable symmetric monoidal stable $\infty$-categories $\mcb{C}$ and $\mcb{D}$.
We recall the \emph{monoidal Barr-Beck theorem} of \cite{MATHEW2017994}, which will be a key technical device to many of the formal results of this paper.

\begin{thm}[{\cite[Thm. 5.29]{MATHEW2017994}}] \label{thm:monoidalBarrBeck} Suppose that $F \dashv R$ satisfies the following conditions:
\begin{enumerate}
    \item $R$ is conservative.
    \item $R$ preserves colimits.
    \item $(F,R)$ satisfies the projection formula: the natural map
    \[ R(X) \otimes Y \to R(X \otimes F(Y)) \]
    is an equivalence for all $X \in \mcb{D}$, $Y \in \mcb{C}$.
\end{enumerate}
Then there is an equivalence
$$\mcb{D} \simeq \Mod_{\mcb{C}}(R(1_{\mcb{D}}))$$ 
and $F \dashv R$ is equivalent to the free-forgetful adjunction. 
\end{thm}

We may descend Thm. \ref{thm:monoidalBarrBeck} to subcategories of local objects.

\begin{lem}\label{lem:PassingLocalizationToModuleCat} Let $\mcb{C}$ and $\mcb{D}$ be presentable symmetric monoidal stable $\infty$-categories, let $$\adjunct{F}{\mcb{C}}{\mcb{D}}{R}$$ 
be a monoidal adjunction, let $E \in \mcb{C}$ be any object and let $E' = F(E)$. Then the adjunction $F \dashv R$ induces a monoidal adjunction $$\adjunct{F'}{\mcb{C}_E}{\mcb{D}_{E'}}{R'}$$ 
between the $\infty$-categories of $E$-local and $E'$-local objects. Moreover:
\begin{enumerate} \item If $R$ is conservative, then $R'$ is conservative. 
 \item Suppose that $(F,R)$ satisfies the projection formula. Then there is an equivalence 
 $$L_E R \simeq R' L_{E'}$$ 
 and $(F',R')$ satisfies the projection formula. Moreover, if $R$ in addition preserves colimits, then $R'$ preserves colimits.
\end{enumerate}
Therefore, we have
$$\Mod_{\mcb{C}}(R(1_{\mcb{D}}))_{E'} \simeq \Mod_{\mcb{C}_E}(L_E R(1_{\mcb{D}}))$$ 
and $F' \dashv R'$ is the free-forgetful adjunction.
\end{lem}

\begin{proof} Because the functor $F$ is strong monoidal, $F$ sends $E$-equivalences to $E'$-equivalences. Therefore, if $X$ is $E'$-local, then $R(X)$ is $E$-local, so we may define 
$$R': \mcb{D}_{E'} \to \mcb{C}_E$$ 
to be the restriction of $R$. We may then define 
$$F': \mcb{C}_E \to \mcb{D}_{E'}$$ 
by $F' \coloneq L_{E'} F$ to obtain the induced monoidal adjunction $$\adjunct{F'}{\mcb{C}_E}{\mcb{D}_{E'}}{R'}.$$

For (1), if $R$ is conservative, then because $i_E R' = R i_{E'}$, $R'$ is conservative. For (2), if $(F,R)$ satisfies the projection formula, then we have that for any $E'$-null object $X$,
\begin{align*} R(X) \otimes E \simeq R(X \otimes E') \simeq 0,
\end{align*}
so $R$ sends $E'$-equivalences to $E$-equivalences. Therefore, we have $L_E R \simeq R' L_{E'}$. To see that $(F',R')$ satisfies the projection formula, we use the sequence of equivalences
\begin{align*} R'(X \otimes_{\mcb{D}_{E'}} F'(Y)) & \simeq  R'(L_{E'} (i_{E'} X \otimes_{\mcb{D}} F(Y) )) \\
& \simeq L_E R (i_{E'} X \otimes_{\mcb{D}} F(Y) )) \\
& \simeq L_E (R(i_{E'} X) \otimes_{\mcb{C}} Y) \\
& \simeq R' X \otimes_{\mcb{C}_E} Y.
\end{align*}
Now suppose that $R$ preserves colimits. To see that $R'$ preserves colimits, suppose given a diagram $X_{\bullet}: J \to \mcb{D}_{E'}$. Then we have
$$\colim X_{\bullet} \simeq L_{E'} \colim i_{E'} X_{\bullet},$$ 
and we have the sequence of equivalences
\begin{align*} R'(\colim X_{\bullet}) & \simeq R' L_{E'} \colim i_{E'} X_{\bullet} \\
& \simeq L_E R \colim i_{E'} X_{\bullet} \\
& \simeq \colim L_E R i_{E'} X_{\bullet} \\
& \simeq \colim R' L_{E'} i_{E'} X_{\bullet} \\
& \simeq \colim R'  X_{\bullet}.
\end{align*}
Finally, the last statement is a consequence of Thm.~\ref{thm:monoidalBarrBeck}.
\end{proof}

\begin{exm} In Lem.~\ref{lem:PassingLocalizationToModuleCat}, let $E = C(x)$ for $x$ as in Example~\ref{ex:Cx}. Then we see that $$\Mod_{\mcb{C}}(R(1_{\mcb{D}}))^{\wedge}_x \simeq \Mod_{\mcb{C}^{\wedge}_x}(R(1_{\mcb{D}})^{\wedge}_x).$$
\end{exm}

We also note a similar result when passing to module categories.

\begin{lem} \label{lem:ModuleProjectionFormula} 
Let $A \in \CAlg(\mcb{C})$ and $A' = F(A)$, and let $$\adjunct{F'}{\Mod_{\mcb{C}}(A)}{\Mod_{\mcb{D}}(A')}{R'}$$ 
denote the induced monoidal adjunction. Then
\begin{enumerate} \item If $R$ is conservative, then $R'$ is conservative.
\item If $R$ preserves colimits, then $R'$ preserves colimits.
\item If $R$ preserves colimits and $(F,R)$ satisfies the projection formula, then $(F',R')$ satisfies the projection formula.
\end{enumerate}
\end{lem}
\begin{proof} Because $F'$ and $R'$ are computed by $F$ and $R$ after forgetting the module structure, the first two results are clear. For the projection formula, under our assumptions the natural map
\[ R M \otimes_A N \to R(M \otimes_{A'} F N) \]
is equivalent to the geometric realization of the map of simplicial diagrams
\[ R M \otimes A^{\otimes \bullet} \otimes N \to R(M \otimes (A')^{\otimes \bullet} \otimes F N), \]
which is an equivalence in view of the projection formula for $(F,R)$.
\end{proof}

\subsection*{Lifting localizations}

For $A \in \CAlg(\mcb{C})$ dualizable and 
$$C(\iota^{n+1}) = \Tot_n \CB^{\bullet}(A)$$ 
as before, let $A' = F(A)$ and $j = F(\iota)$, so that $F(C(\iota^{n})) \simeq C({j}^n)$. We have induced monoidal adjunctions
\begin{gather*}
\adjunct{F_n}{\Mod_{\mcb{C}}(C(\iota^n))}{\Mod_{\mcb{D}}(C({j}^n))}{R_n}, \\
\adjunct{F_{\infty}}{\mcb{C}_A}{\mcb{D}_{A'}}{R_{\infty}}.
\end{gather*}

We end this section with a result that allows us to lift the property of $R_1$ being fully faithful to $R_n$ and $R_{\infty}$. 

\begin{prop} \label{prop:FullyFaithfulLift} Suppose that $R$ preserves colimits, $(F,R)$ satisfies the projection formula, and $R_1$ is fully faithful. Then $R_n$ is fully faithful for all $1 \leq n \leq \infty$.
\end{prop}

\begin{proof} First suppose $n < \infty$ and let $X \in \Mod_{\mcb{D}}(C({j}^n))$. We need to prove that the counit $\epsilon^n_X$ is an equivalence. Because the base-change functor $- \otimes_{C {j}^n} A'$ is conservative by Lem.~\ref{lem:BaseChangeConservative}, it suffices to show that $\epsilon^n_X \otimes_{C{j}^n} A'$ is an equivalence. But by the projection formula for $(F_n, R_n)$ established in Lem.~\ref{lem:ModuleProjectionFormula}, this map is equivalent to the counit $\epsilon^n_{(X \otimes_{C {j}^n} A')}$. Because $X \otimes_{C {j}^n} A'$ is an $A'$-module, $\epsilon^n_{(X \otimes_{C {j}^n} A')}$ is lifted by the counit $\epsilon^1_{(X \otimes_{C {j}^n} A')}$, which is an equivalence by assumption. The proof for the case $n = \infty$ is similar, where we instead use that 
$$- \otimes A': \mcb{D}_{A'} \to \Mod_{\mcb{D}}(A')$$ 
is conservative and the projection formula for $(F_{\infty}, R_{\infty})$ by Lem.~\ref{lem:PassingLocalizationToModuleCat}. 
\end{proof}

\section{Cellularization}\label{sec:cellularization}

In this section, we collect a few technical lemmas that will be applied to study the $\infty$-category $\SH_{\cell}(S)$ of cellular motivic spectra. To begin with, we have the following variant of \cite[Prop. 2.27]{MATHEW2017994} (with the same conclusion), where we do not assume that $E$ is an algebra object of $\cC$.

\begin{lem} \label{lm:CompactGenerationAndBousfieldLocalization} 
Suppose $\cC$ is a presentable symmetric monoidal stable $\infty$-category and $E$ is a dualizable object in $\cC$.
\begin{enumerate} 
\item  For any object $X \in \cC$, 
$E^{\vee} \otimes X$ 
is $E$-local. If $E^{\vee} \simeq E \otimes \kappa$, then 
$E \otimes X$ 
is also $E$-local.
\vspace{10pt}

\item For any compact object $X \in \cC$, $E^{\vee} \otimes X$ is compact in $\cC_E$. If $E^{\vee} \simeq E \otimes \kappa$, then $E \otimes X$ is also compact in $\cC_E$.
\vspace{10pt}

\item If $\{ X_i \}$ is a set of compact generators of $\cC$, then $\{ E^{\vee} \otimes X_i \}$ is a set of compact generators of $\cC_E$. Therefore, if $\cC$ is compactly generated, then $\cC_E$ is compactly generated.
\end{enumerate}
\end{lem}

\begin{proof} (1) Let $Z$ be an $E$-null object. Then
\[ \Hom_\cC(Z,E^{\vee} \otimes X) \simeq \Hom_\cC(Z \otimes E, X) \simeq 0, \]
so $E^{\vee} \otimes X$ is $E$-local. If $E^{\vee} \simeq E \otimes \kappa$, then
\[ \Hom_\cC(Z, E \otimes X) \simeq \Hom_\cC(Z \otimes E^{\vee}, X) \simeq \Hom_C(Z \otimes E \otimes \kappa, X) \simeq 0, \]
so $E \otimes X$ is $E$-local.

(2) Observe that the functor 
$$ \cC_E \to \cC $$ 
given by $Y \mapsto Y \otimes E$ preserves colimits (\cite[2.20]{MATHEW2017994}). Let $$Y_{\bullet}: J \to \cC_E$$ 
be a functor and let us write $\colim Y_j$ for the colimit in $\cC$ and $L_E(\colim Y_j)$ for the colimit in $\cC_E$. Then we have
\begin{align*} \Hom_\cC(E^{\vee} \otimes X, L_E(\colim Y_j)) 
& \simeq \Hom_\cC(X, E \otimes L_E(\colim Y_j)) \\ 
& \simeq \Hom_\cC(X, \colim (E \otimes Y_j)) \\
& \simeq \colim \Hom_\cC(X, E \otimes Y_j) \\
& \simeq \colim \Hom_\cC(E^{\vee} \otimes X, Y_j),
\end{align*}
so $E^{\vee} \otimes X$ is compact. The second assertion is similar.

(3) This follows as in the proof of \cite[2.27]{MATHEW2017994}. 
\end{proof}

The following result concerns the existence and basic properties of cellularization.

\begin{lem} \label{lem:ExistenceOfCellularization} Let $\cC$ be a compactly generated stable $\infty$-category, let $S = \{ S_i : i \in \mc{I} \}$ be a small set of compact objects in $\cC$, and let $\cC'$ be the localizing subcategory generated by $S$ (i.e., the smallest full stable subcategory containing $S$ that is closed under colimits).
\begin{enumerate} \item $\cC'$ is compactly generated and is a coreflective subcategory of $\cC$ (i.e., the inclusion 
$j: \cC' \subseteq \cC$ admits a right adjoint). Moreover, if 
$$\Cell: \cC \to \cC'$$ 
denotes this right adjoint, then $\Cell$ also preserves colimits.
\vspace{10pt}

\item Suppose in addition that $\cC$ is closed symmetric monoidal, the unit $1 \in \cC$ is compact and in $S$, and for all $i, i' \in \mc{I}$, we have that $S_i \otimes S_{i'} \in S$. Then $\cC' \subseteq \cC$ is a symmetric monoidal subcategory.
\vspace{10pt}

\item Suppose in addition to the assumptions of (2) that each $S_i$ is dualizable. Then for all $X \in \cC$ and $Y \in \cC'$, the natural map $\theta: \Cell(X) \otimes Y \to \Cell(X \otimes Y)$ is an equivalence.
\end{enumerate}
\end{lem}

\begin{proof} For (1), $\cC'$ is compactly generated by definition, so $j$ admits a right adjoint by the adjoint functor theorem \cite[5.5.2.9]{HTT}. Moreover, the set $S$ furnishes a set of compact generators for $\cC'$ that are sent to compact objects under $j$, so $\Cell$ preserves colimits. For (2), because the tensor product $\otimes$ commutes with colimits separately in each variable, our assumption ensures that if $X, Y \in \cC'$ then $X \otimes Y \in \cC'$. We may then invoke \cite[2.2.1.2]{HA} to see that $\cC' \subseteq \cC$ is a symmetric monoidal subcategory. For (3), the assumption ensures that $\cC'$ is generated by dualizable objects under colimits. Because $\Cell$ commutes with colimits, both the source and target of $\theta$ commute with colimits separately in each variable. We may thus suppose that $Y$ is a dualizable object in $\cC'$, with dual $Y^\vee$. Note that $Y^\vee$ is also the dual of $Y$ in $\cC$.  For each generator $S_i$ we have that
\begin{align*} 
\Hom_\cC(S_i, \Cell(X) \otimes Y) 
& \simeq \Hom_\cC(S_i \otimes Y^{\vee}, \Cell(X)) \\
& \simeq \Hom_\cC(S_i \otimes Y^{\vee}, X) \\
& \simeq \Hom_\cC(S_i, X \otimes Y) \\
& \simeq \Hom_\cC(S_i, \Cell(X \otimes Y)),
\end{align*}
so $\theta$ is an equivalence.
\end{proof}

The following two lemmas describe the interaction of cellularization with Bousfield localization and passage to module categories.

\begin{lem} \label{lm:BousfieldLocalizationOfCellular} With the setup of Lem.~\ref{lem:ExistenceOfCellularization}(2), let $E$ be a dualizable object in $\cC'$. Then:
\begin{enumerate}
    \item If $X \in \cC$ is $j(E)$-local, then $\Cell(X) \in \cC'$ is $E$-local.
    \vspace{10pt}
    
    \item For $X \in \cC$, the natural map 
    $$\Cell (X) \otimes E \to \Cell(X \otimes j(E))$$ 
    is an equivalence. Consequently, $\Cell$ sends $j(E)$-equivalences to $E$-equivalences.
    \vspace{10pt}

    \item The adjunction 
    $$\adjunct{j}{\cC'}{\cC}{\Cell}$$ 
    induces a monoidal adjunction 
    $$\adjunct{j'}{\cC'_E}{\cC_{j(E)}}{\Cell'}$$ 
    such that $\Cell'(X) \simeq \Cell(X)$ for $X \in \cC_{j(E)}$, $j'(Y) \simeq L_{j(E)} j(Y)$ for $Y \in \cC'_E$, and the functor $j'$ is fully faithful.
    \vspace{10pt}
    
    \item The functor $\Cell'$ preserves colimits.
    \vspace{10pt}
    
    \item Suppose in addition the condition of Lem.~\ref{lem:ExistenceOfCellularization}(3). Then for all $X \in \cC_{j(E)}$ and $Y \in \cC'_E$, we have the natural equivalence $$L_E(\Cell'(X) \otimes Y) \simeq \Cell'(L_E(X \otimes Y)).$$ Consequently, the conclusion of Lem.~\ref{lm:CellularProjectionFormula} holds with $j' \dashv \Cell'$ in place of $j \dashv \Cell$.
\end{enumerate}
\end{lem}

\begin{proof} We consider each assertion in turn:

\begin{enumerate} \item If $Y \in \cC'$ is $E$-null, then $j(Y) \in \cC$ is $j(E)$-null since the inclusion $\cC' \subseteq \cC$ is strong monoidal. Then if $X \in \cC$ is $j(E)$-local, we have for all $Y \in \cC'$ $E$-null that $\Hom_\cC(Y, \Cell X) \simeq \Hom_\cC(Y, X) \simeq 0$, so $\Cell(X)$ is $E$-local.
\vspace{10pt}

\item We write $j(E)$ as $E$ for clarity. It suffices to observe that for all $i \in \mc{I}$,
\begin{align*}  \Hom_\cC(S_i, \Cell (X) \otimes E) 
& \simeq \Hom_\cC(S_i \otimes E^{\vee}, \Cell (X)) \\ 
& \simeq \Hom_\cC(S_i \otimes E^{\vee}, X) \\
& \simeq \Hom_\cC(S_i, X \otimes E ) \\
& \simeq \Hom_\cC(S_i, \Cell(X \otimes E)).
\end{align*}
\item By (1), $\Cell: \cC \to \cC'$ restricts to a functor 
$$\Cell': \cC_{j(E)} \to \cC'_E. $$ 
Define 
$$j': \cC'_E \to \cC_{j(E)}$$ 
to be the composite 
$$\cC'_E \subseteq \cC \xto{L_{j(E)}} \cC_{j(E)}.$$ 
Then it is clear that $j' \dashv \Cell'$, the adjunction is monoidal with respect to the tensor products $L_E(- \otimes -)$ and $L_{j(E)}(- \otimes -)$ on $\cC'_E$ and $\cC_{j(E)}$, and the unit map 
$$\eta: Y \to \Cell' j' Y$$ 
is equivalent to $\Cell$ of the unit map 
$$\widehat{\eta}: Y \to L_{j(E)} Y.$$
Because $\widehat{\eta}$ is an $j(E)$-equivalence in $\cC$, by (2) we see that $\Cell(\widehat{\eta})$ is an equivalence.
\vspace{10pt}

\item By Lem.~\ref{lm:CompactGenerationAndBousfieldLocalization}, $\{ S_i \otimes E^{\vee} : i \in \mc{I} \}$ are a set of compact generators for $\cC'_E$, and are also compact and $j(E)$-local objects when regarded as being in $\cC$. Therefore, the left adjoint $j'$ sends compact generators to compact objects, which implies that the right adjoint $\Cell'$ preserves colimits.
\vspace{10pt}

\item With our additional assumption, the $S_i \otimes E^{\vee}$ constitute a set of compact dualizable generators of $\cC'_E$. The proof of 
Lem.~\ref{lem:ExistenceOfCellularization}(3) then applies to $j' \dashv \Cell'$.
\end{enumerate}
\end{proof}

\begin{lem} \label{lem:CellularizationOnModules} With the setup of 
Lem.~\ref{lem:ExistenceOfCellularization}(3), let $A$ be an $E_{\infty}$-algebra in $\cC$ and let $A' \coloneq \Cell(A)$ be the resulting $E_{\infty}$-algebra in $\cC'$. Then we have an induced adjunction
\[ \adjunct{j'}{\Mod_{\cC'}(A')}{\Mod_\cC(A)}{\Cell'} \]
such that $j'$ is fully faithful and identifies $\Mod_{\cC'}(A')$ with the localizing subcategory of $\Mod_\cC(A)$ generated by $S_A \coloneq \{ S_i \otimes A: i \in \mc{I}\}$.
\end{lem}

\begin{proof} Note that $\Mod_\cC(A)$ and $\Mod_{\cC'}(A')$ are compactly generated stable symmetric monoidal $\infty$-categories, and the set $S_{A'} \coloneq \{ S_i \otimes A' \}$ furnishes a set of compact dualizable generators for $\Mod_{\cC'}(A')$. Because $\Cell$ is lax monoidal, it induces a functor $\Cell': \Mod_\cC(A) \to \Mod_{\cC'}(A')$ such that the diagram of right adjoints
\[ \begin{tikzcd}[row sep=4ex, column sep=4ex, text height=1.5ex, text depth=0.25ex]
\cC' & \cC \ar{l}{\Cell} \\
\Mod_{\cC'}(A') \ar{u}{U'} & \Mod_{\cC}(A) \ar{l}{\Cell'} \ar{u}{U}
\end{tikzcd} \]
commutes (where $U$ and $U'$ denote the forgetful functors). Since $\Cell$ preserves limits and $U, U'$ create limits, $\Cell'$ also preserves limits and therefore admits a left adjoint $j'$ such that the diagram of left adjoints
\[ \begin{tikzcd}[row sep=4ex, column sep=4ex, text height=1.5ex, text depth=0.25ex]
\cC' \ar{d}{F'} \ar{r}{j} & \cC \ar{d}{F}  \\
\Mod_{\cC'}(A') \ar{r}{j'} & \Mod_{\cC}(A)
\end{tikzcd} \]
commutes (where $F$ and $F'$ denote the free functors), so $j(S_{A'}) = S_A$. It remains to show that $j'$ is fully faithful, i.e., that the unit map 
$$\eta: M \to \Cell' j' M$$ 
is an equivalence for all $M \in \Mod_{\cC'}(A')$. For this, note that $\Cell'$ preserves colimits since $\Cell$ preserves colimits by 
Lem.~\ref{lem:ExistenceOfCellularization}(2) and $U, U'$ create colimits, so we may suppose that $M = S_i \otimes A'$. But then we have 
$$\Cell' j' (S_i \otimes A') = \Cell' (S_i \otimes A) \simeq S_i \otimes A'$$ 
by Lem.~\ref{lem:ExistenceOfCellularization}(3), and it is easily checked that $\eta$ implements this equivalence.
\end{proof}

Finally, we retain the projection formula after cellularization.

\begin{lem} \label{lm:CellularProjectionFormula} With the setup of Lem.~\ref{lem:ExistenceOfCellularization}(3), let $\cD$ be a presentable symmetric monoidal stable $\infty$-category and let 
$$\adjunct{F}{\cC}{\cD}{R}$$ 
be a monoidal adjunction such that $R$ preserves colimits and $(F,R)$ satisfies the projection formula. Then $\Cell R$ preserves colimits and $(F j, \Cell R)$ satisfies the projection formula.
\end{lem}

\begin{proof} $\Cell R$ preserves colimits by Lem.~\ref{lem:ExistenceOfCellularization}(2). For the projection formula, we note that for all $X \in \cD$ and $Y \in \cC'$,
\begin{align*} (\Cell R)(X) \otimes Y \simeq \Cell(R X \otimes Y) \simeq (\Cell R)(X \otimes L Y),
\end{align*}
where the first equivalence is by Lem.~\ref{lem:ExistenceOfCellularization}(3) and the second by our assumption on $(L,R)$.
\end{proof}

\section{Recollements}\label{sec:recollements}

Let $\mc{X}$ be an $\infty$-category which admits finite limits.
Recall \cite[A.8]{HA}, \cite{BarwickGlasman} that an $\infty$-category $\mc{X}$ is a \emph{recollement} of two full subcategories $\mc{U}$ and $\mc{Z}$ if the inclusions $j_*$, $i_*$ of these subcategories admit left adjoints $j^*$, $i^*$:
\[ \begin{tikzcd}[row sep=4ex, column sep=4ex, text height=1.5ex, text depth=0.25ex]
\cU \ar[shift right=1]{r}[swap]{j_{\ast}} & \cX \ar[shift left=1]{r}{i^{\ast}} \ar[shift right=1]{l}[swap]{j^{\ast}} & \cZ \ar[shift left=1]{l}{i_{\ast}}
\end{tikzcd} \]
such that:
\begin{enumerate}
\item The subcategories $\mc{U}, \mc{Z} \subseteq \mc{X}$ are stable under equivalence.
\item The left adjoints $j^*$, $i^*$ are left exact.
\item The functor $j^{\ast} i_{\ast}$ is equivalent to the constant functor at the terminal object. 
\item If $f$ is a morphism of $\mc{X}$ such that $j^*f$ and $i^*f$ are equivalences, then $f$ is an equivalence.
\end{enumerate}
  
The following lemma shows that if $\mc{X}$ is a recollement of $\mc{U}$ and $\mc{Z}$, then to test whether a functor into $\mc{X}$ is a localization, it suffices to check this on $\mc{U}$ and $\mc{Z}$.

\begin{lem} \label{lem:RecollementFullyFaithful} Let $\cC$ and $\cX$ be $\infty$-categories that admit finite limits and suppose that we have a recollement on $\cX$
\[ \begin{tikzcd}[row sep=4ex, column sep=4ex, text height=1.5ex, text depth=0.25ex]
\cU \ar[shift right=1]{r}[swap]{j_{\ast}} & \cX \ar[shift left=1]{r}{i^{\ast}} \ar[shift right=1]{l}[swap]{j^{\ast}} & \cZ \ar[shift left=1]{l}{i_{\ast}}
\end{tikzcd} \]
and an adjunction $\adjunct{F}{\cC}{\cX}{R}$ with $F$ also left exact such that
\begin{enumerate}
    \item The natural transformation $i^{\ast} F R j_{\ast} \Rightarrow i^{\ast} j_{\ast}$ induced by the counit of $(F,R)$ is an equivalence.
    \item The functor $j^{\ast} F R i_{\ast}$ is equivalent to the constant functor at the terminal object. 
    \item The two functors $R j_{\ast}$ and $R i_{\ast}$ are fully faithful.
\end{enumerate}
Then $R$ is fully faithful.
\end{lem}
\begin{proof} We will show that for any $X \in \cX$, the counit $\epsilon: F R X \to X$ is an equivalence. Because $i^{\ast}$ and $j^{\ast}$ are jointly conservative, it suffices to show that $i^{\ast} \epsilon$ and $j^{\ast} \epsilon$ are equivalences. Consider the pullback square
\[ \begin{tikzcd}[row sep=4ex, column sep=4ex, text height=1.5ex, text depth=0.25ex]
X \ar{r} \ar{d} & i_{\ast} i^{\ast} X \ar{d} \\
j_{\ast} j^{\ast} X \ar{r} & i_{\ast} i^{\ast} j_{\ast} j^{\ast} X.
\end{tikzcd} \]
Applying $i^{\ast} F R$ and using that $R i_{\ast}$ is fully faithful and $i^{\ast} F R j_{\ast} \simeq i^{\ast} j_{\ast}$, we obtain a pullback square
\[ \begin{tikzcd}[row sep=4ex, column sep=4ex, text height=1.5ex, text depth=0.25ex]
i^{\ast} F R X \ar{r}{\simeq} \ar{d} & i^{\ast} F R i_{\ast} i^{\ast} X \simeq i^{\ast} X \ar{d} \\
i^{\ast} F R j_{\ast} j^{\ast} X \simeq i^{\ast} j_{\ast} j^{\ast} X \ar{r}{\simeq} & i^{\ast} F R i_{\ast} i^{\ast} j_{\ast} j^{\ast} X \simeq i^{\ast} j_{\ast} j^{\ast} X
\end{tikzcd} \]
from which it follows that $i^{\ast} \epsilon$ is an equivalence. Applying $j^{\ast} F R$ and using that $R j_{\ast}$ is fully faithful and $j^{\ast} F R i_{\ast} \simeq 0$, we obtain a pullback square
\[ \begin{tikzcd}[row sep=4ex, column sep=4ex, text height=1.5ex, text depth=0.25ex]
j^{\ast} F R X \ar{r} \ar{d}{\simeq} & j^{\ast} F R i_{\ast} i^{\ast} X \simeq 0 \ar{d}{\simeq} \\
j^{\ast} F R j_{\ast} j^{\ast} X \simeq j^{\ast} X \ar{r} & j^{\ast} F R i_{\ast} i^{\ast} j_{\ast} j^{\ast} X \simeq 0
\end{tikzcd} \]
from which it follows that $j^{\ast} \epsilon$ is an equivalence.
\end{proof}

\section{Background on motivic and equivariant homotopy theory}\label{sec:background}

\subsection*{The motivic stable homotopy category}

Let $S$ be a scheme and let $\SH(S)$ denote the symmetric monoidal $\infty$-category of motivic $\PP^1$-spectra over $S$. Let $\SH_{\cell}(S)$ be the localizing subcategory of $\SH(S)$ generated by the motivic spheres $\{S^{p,q} \}$. A motivic spectrum $E$ is \emph{cellular} if it lies inside $\SH_{\cell}(S)$. Note that the full hypotheses of Lem. \ref{lem:ExistenceOfCellularization} apply.

We recall from \cite[\S 2.2]{ElmantoKolderup} the following facts concerning compact and dualizable objects and generation in $\SH(S)$:
    \begin{enumerate}
        \item For $X$ an affine smooth scheme over $S$ and $q \in \mb{Z}$, the motivic $\PP^1$-spectrum $\Sigma^{q} X_+$ is compact; in particular, the bigraded motivic spheres $S^{p,q}$ are compact. Compactness of the unit then implies that every dualizable object in $\SH(S)$ is compact. Moreover, $\SH(S)$ is generated under sifted colimits by $\Sigma^{q} X_+$ and is thus compactly generated.
        \vspace{10pt}
        
        \item If $K$ is a field of characteristic $0$, then every compact object in $\SH(K)$ is dualizable.
    \end{enumerate}

We collect a few facts concerning the functoriality of $\SH(-)$; see \cite{HOYOIS2017197} for a reference. Let $f: T \to S$ be a morphism of schemes. We always have a monoidal adjunction
\[ \adjunct{f^\ast}{\SH(S)}{\SH(T)}{f_\ast}. \]
The left adjoint $f_\sharp$ to $f^\ast$ exists if $f$ is smooth. If $f$ is smooth and proper, we have the duality equivalence
\[ f_\ast \simeq f_\sharp \Sigma^{-\Omega_f}. \]
In particular, if $f$ is finite etale, then $f_\ast \simeq f_\sharp$ and the adjunction $f^{\ast} \dashv f_{\ast}$ is ambidextrous. On the other hand, if $f$ is separated and of finite type, we have the adjunction
\[ \adjunct{f_!}{\SH(T)}{\SH(S)}{f^!}. \]
Moreover, $f_!$ coincides with $f_\ast$ if $f$ is proper. If $f$ is finite etale, we have that $f^! \simeq f^\ast$. Finally, we have the projection formula
\[ f_!(X \wedge f^\ast(Y)) \simeq f_!(X) \wedge Y. \]

\subsection*{Euler classes}

Let  
$$\rho = \rho_S: S^{-1,-1} \to S^{0,0} $$
be the map in $\SH(S)$ induced by the inclusion 
$$S^{0,0} = \{\pm 1\} \righthookto \GG_m = S^{1,1}. $$ 

The equivariant analog is the element $a \in \pi_{-1,-1}^{C_2}S$ 
induced by the inclusion
$$ S^0 \righthookto S^\sigma. $$
The element $a$ is the $C_2$-Betti realization of the element $\rho \in \pi^\RR_{-1,-1}$, and also serves as the Euler class for the representation $\sigma$. 

For $Y \in \Sp^{C_2}$, the cofiber sequence 
\begin{equation}\label{eq:cofiber}
 \Sigma^\infty_+C_2 \to S^0 \xto{\Sigma^\sigma a} S^\sigma
\end{equation}
yields a long exact sequence
$$ \cdots \to \pi^{C_2}_{i+1, 1}Y \xto{a} \pi^{C_2}_{i} Y \to \pi^e_{i}Y \to \cdots. $$
It follows that a map of $C_2$-spectra is a stable equivalence if and only if it induces a an isomorphism on the bigraded homotopy groups $\pi^{C_2}_{*,*}$, and that, in contrast to the $\RR$-motivic case, every $C_2$-spectrum is stably equivalent to one built from representation spheres.

 The cofiber sequence (\ref{eq:cofiber}) results in an equivalence
\begin{equation}\label{eq:CaC2}
 Ca \simeq \Sigma^{1-\sigma}\Sigma^\infty_+ C_2.
\end{equation}

More generally, the cofiber sequences
$$ S(i\sigma)_+ \to S^0 \xto{\Sigma^{i\sigma} a^i} S^{i\sigma} $$
yield equivalences
\begin{equation}\label{eq:Caiequiv}
 Ca^i \simeq \Sigma^{1-i\sigma} S(i\sigma)_+
 \end{equation}
and, taking Spanier-Whitehead duals, equivalences
$$ Ca^i \simeq (S(i\sigma)_+)^\vee. $$
We therefore have, for any $Y \in \Sp^{C_2}$
\begin{equation}\label{eq:acompletion}
\begin{split}
Y^h & = F((EC_2)_+, Y) \\
& \simeq \lim_i F(S(i\sigma)_+,Y) \\
& \simeq \lim_i Y \wedge C(a^i) \\
& \simeq  Y^\wedge_a.
\end{split}
\end{equation}
Since we have
\begin{equation}\label{eq:phidetails}
\begin{split} 
Y^{\Phi} & = Y \wedge \td{EC_2} \\
& \simeq \colim_i Y \wedge S^{i\sigma} \\
& \simeq Y[a^{-1}]
\end{split}
\end{equation}
we deduce that the isotropy separation square (\ref{eq:Tate}) is equivalent to the $a$-arithmetic square:
$$
\xymatrix{
Y \ar[r] \ar[d] & Y[a^{-1}] \ar[d] \\
Y^\wedge_a \ar[r] & Y^\wedge_a[a^{-1}]
}
$$
Therefore, $C_2$-Betti realization takes the $\rho$-arithmetic square 
to the isotropy separation square.

\subsection*{$\eta$-completion and $\eta$-localization at odd primes}

Let $K$ be a perfect field.  In \cite[Lem.~39]{Bachmann}, Bachmann summarizes relations in $\pi^K_{*,*}S^{0,0}$ involving the Hopf map
$$ \eta \in \pi^K_{1,1}S^{0,0} $$
and the element $\rho \in \pi_{-1,-1}^K S^{0,0}$, after $2$ is inverted.  Namely, the element\footnote{Here we are following the convention that $\rho = [-1]$.  Bachmann instead takes $\rho = -[-1]$ which results in the formula $\epsilon = \eta\rho - 1$ in his work.}
$$ \epsilon := -\eta\rho - 1 $$
is the interchange isomorphism
$$ \epsilon: S^{1,1} \wedge S^{1,1} \to S^{1,1} \wedge S^{1,1}. $$
Therefore it satisfies $\epsilon^2 \simeq 1$, and hence for any $X \in \SH(K)$[1/2], there is a corresponding decomposition into $\pm 1$-eigenspaces
\begin{equation}\label{eq:epsilondecomp}
 \pi^K_{*,*}X \cong \pi^K_{*,*}X^- \oplus \pi^K_{*,*}X^+.
 \end{equation}
Here $(-)^-$ is the $+1$ eigenspace, and $(-)^+$ is the $-1$ eigenspace.  We have
$$ \pi^K_{*,*}X^- = \pi^K_{*,*}X[\eta^{-1}] = \pi^K_{*,*}X[\rho^{-1}] $$
and on $\pi^K_{*,*}X^+$ multiplication by $\eta$ and $\rho^2$ is 
zero.\footnote{When $K = \RR$, multiplication by $\rho$ is zero on $\pi^\RR_{*,*}X^+$.  This follows from the presentation of the Milnor-Witt ring of $\RR$ in the introduction of \cite{DuggerIsaksenMW}.}
We deduce the following proposition. 

\begin{prop}\label{prop:eta}
For any $X \in \SH_{\cell}(K)[1/2]$, we have
$$ X[\eta^{-1}] \simeq X[\rho^{-1}] $$
and the homotopy groups of these spectra are $\pi^K_{*,*}X^-$, and we have
$$ X^{\wedge}_\eta \simeq X^\wedge_\rho $$
and the homotopy groups of these spectra are $\pi^K_{*,*}X^+$.
\end{prop}

\begin{proof}
From the discussion above we deduce that the maps
\begin{gather*}
X[\eta^{-1}] \rightarrow X[\rho^{-1}][\eta^{-1}] \leftarrow X[\rho^{-1}] \\
X^{\wedge}_\eta \rightarrow X^{\wedge}_{\rho, \eta} \leftarrow X^{\wedge}_\eta
\end{gather*}
induces isomorphisms on bigraded homotopy groups, and hence are equivalences since the spectra are cellular.
\end{proof}

Finally we note that for $X \in \SH_{\cell}(K)[1/2]$, since $X^{\wedge}_\rho[\rho^{-1}] \simeq 0$, the $\rho$-arithmetic square
$$
\xymatrix{
X \ar[r] \ar[d] & X[\rho^{-1}] \ar[d] \\
X^\wedge_\rho \ar[r] & X^\wedge_\rho[\rho^{-1}] 
}
$$ 
yields a topological lift of the decomposition (\ref{eq:epsilondecomp})
\begin{equation}\label{eq:rhosplitting}
X \simeq X[\rho^{-1}] \vee X^\wedge_\rho.
\end{equation}

On the other hand, for any $Y \in \Sp^{C_2}[1/2]$, the Tate spectrum $Y^t$ is contractible, and the isotropy separation square reduces to a splitting
\begin{equation}\label{eq:asplitting}
Y \simeq Y^\Phi \vee Y^h.
\end{equation}
The discussion from the previous subsection implies that $C_2$-Betti realization carries the splitting (\ref{eq:rhosplitting}) to (\ref{eq:asplitting}).

\subsection*{Motivic and equivariant cohomology}

Let $(H\FF_p)_K$ denote the mod $p$ motivic Eilenberg-MacLane spectrum over $K$.  We have by \cite{Voevodsky1}, \cite{Voevodsky2}, \cite{Stahn}
$$
\pi^K_{*,*}(H\FF_p)_K = \begin{cases}
\FF_p[\tau], & K = \CC, \\
\FF_2[\tau, \rho], & K = \RR, \: p = 2, \\
\FF_p[\tau^2], & K = \RR, \: p \: \mr{odd}.
\end{cases}
$$
Here, $\rho$ is the Hurewicz image of $\rho_\RR$ (and $\rho_\CC \simeq 0$).

Eilenberg MacLane spectra are stable under base change --- in particular, we have 
$$ \zeta^* (H\FF_p)_\RR = (H\FF_p)_\CC $$
and the associated map
$$ \pi^{\RR}_{*,*}(H\FF_p)_\RR \to \pi^{\CC}_{*,*}(H\FF_p)_\CC $$
is the quotient by the ideal generated by $\rho$ if $p = 2$, and the evident inclusion if $p$ is odd.


The $C_2$-Betti realization of the mod $p$ motivic Eilenberg-MacLane spectrum is the $C_2$-equivariant Eilenberg-MacLane spectrum $H\ul{\FF_p}$ associated to the constant Mackey functor $\ul{\FF_p}$ \cite{HellerOrmsby}:
$$ \Be^{C_2}(H\FF_p)_\RR \simeq H\ul{\FF_p}. $$ 
For $p = 2$ we have
$$ \pi^{C_2}_{*,*}H\ul{\FF_2} = \FF_2[u,a] \oplus \frac{\FF_2[u,a]}{(u^\infty, a^\infty)}\{ \theta \} $$
where $a$ is the Hurewicz image of the element $a \in \pi^{C_2}_{-1,-1}$, 
$$ u = \Be^{C_2}(\tau) \in \pi^{C_2}_{0,-1}H\ul{\FF_2}, $$
and
$$ \theta \in \pi_{0,2}^{C_2}H\ul{\FF_2}. $$
For $p$ odd we have
$$ \pi^{C_2}_{*,*} H\ul{\FF_p} = \FF_p[u^{\pm 2}] $$
where
$$ u^2 = \Be^{C_2}(\tau^2) \in \pi^{C_2}_{0,-2}H\ul{\FF_p}. $$

\section{\texorpdfstring{$\tau$}{tau}-self maps}\label{sec:tau}

In this section we will construct $\tau^j$-self maps on the spectra $C(\rho^i)^{\wedge}_p$.  For $p = 2$, this will be accomplished in the first three subsections by first constructing the $C_2$-Betti realizations of the desired self-maps, and then by using a theorem of Dugger-Isaksen \cite{DuggerIsaksenC2} to lift these equivariant self maps to real motivic self maps.  For $p$ odd, we will observe in the last subsection that the work of Stahn \cite{Stahn} implies that every $\rho$-complete spectrum has a $\tau^2$-self map.

\emph{From now until the last subsection of this section, we implicitly assume everything is $2$-complete.}

\subsection*{$\RR$-motivic and $C_2$-equivariant homotopy groups of spheres}

For $j \in \ZZ$, let $P_j^\infty$ denote the stunted projective spectrum given as the Thom spectrum
$$ P_j^\infty := (\RR P^\infty)^{j\xi} $$
where $\xi$ is the canonical line bundle.  
The Segal conjecture for the group $C_2$ (Lin's theorem) \cite{Lin} implies the following.

\begin{prop}
There are isomorphisms
$$
\pi^{C_2}_{i,j}S^{0,0} \cong \pi_{i-j}([P^\infty_j]^{\vee}).
$$
\end{prop}

\begin{proof}
The Segal conjecture implies that for a finite $C_2$-spectrum $Y$, the map
$$ Y \to Y^h = F((EC_2)_+, Y) $$
is a ($2$-adic) equivalence. 
Using the equivalence
$$ P^\infty_j \simeq (S^{j\sigma})_{hC_2}, $$
we have
\begin{align*}
\pi^{C_2}_{i,j} & = [S^{i-j} \wedge S^{j\sigma}, S]^{C_2} \\
& \cong [S^{i-j} \wedge S^{j\sigma}, F((EC_2)_+, S)]^{C_2} \\
& \cong [S^{i-j}, F((EC_2)_+ \wedge S^{j\sigma}, S)]^{C_2} \\
& \cong [S^{i-j}, F((EC_2)_+ \wedge_{C_2} S^{j\sigma}, S)] \\
& = \pi_{i-j}([P^\infty_j]^\vee).
\end{align*}
\end{proof}

Applying $\pi^{C_2}_{*,*}$ to the norm cofiber sequence
\begin{equation}\label{eq:normcofiber}
 (EC_2)_+ \to S^0 \to \td{EC}_2
\end{equation}
gives the long exact sequence
$$ \cdots \to \pi^s_{i-j+1} \to \lambda_{i,j} \to \pi^{C_2}_{i,j} \xto{\Phi^{C_2}} \pi^s_{j-i} \to \cdots $$
studied by
Landweber\footnote{Here, we have indexed $\pi_{i,j}$ and $\lambda_{i,j}$ with respect to our bigrading convention, not Landweber's.} \cite{Landweber}.
Using the equivalences
\begin{align*}
[P^\infty_j]^\vee & \simeq \Sigma P^{-j-1}_{-\infty}, && \text{\cite[Thm.~V.2.14 (iv) ]{Hinfty}} \\
S^{-1} & \simeq P^{\infty}_{-\infty}, && \text{\cite{Lin}}
\end{align*}
there is an isomorphism of long exact sequences
\begin{equation}\label{eq:LEScompare}
\xymatrix@C-1em{
\cdots \ar[r] & \pi^s_{i-j+1} \ar[r] \ar[d]^{\cong} & \lambda_{i,j} \ar[r]^{\cong} \ar[d]^\cong & \pi^{C_2}_{i,j} \ar[r]^{\Phi^{C_2}} \ar[d]^{\cong} & \pi^s_{j-i} \ar[r] \ar[d]^{\cong} & \cdots
\\
 \cdots \ar[r] & \pi_{i-j} P^\infty_{-\infty} \ar[r]  & \pi_{i-j}P^{\infty}_{-j} \ar[r] & \pi_{i-j-1}P^{-j-1}_{-\infty} \ar[r]  & \pi_{i-j-1} P^{\infty}_{-\infty} \ar[r] & \cdots. 
}
\end{equation}
where the bottom long exact sequence is the sequence obtained by applying $\pi_*$ to the cofiber sequence
$$ P_{-\infty}^{-j-1} \to P^\infty_{-\infty} \to P^{\infty}_{-j}. $$
By (\ref{eq:phidetails}), the geometric fixed points map is the $a$-localization map:
$$
\xymatrix{
\pi^{C_2}_{i,j}S \ar[r] \ar[dr]_{\Phi^{C_2}} & \pi^{C_2}_{i,j}S[a^{-1}] \ar[d]^{\cong} \\
& \pi_{i-j} S
}
$$
Thus the groups $\pi^{C_2}_{*,*}$ consist of $a$-torsion, and $a$-towers, where the latter are in bijective correspondence with the non-equivariant stable stems.  The generators of these $a$-towers correspond to the Mahowald invariants \cite{BrunerGreenlees}. 

As explained in \cite{DuggerIsaksenC2}, Landweber \cite{Landweber} uses James periodicity to show that the $a$-torsion in $\pi^{C_2}_{i,j}$ is periodic in the $j$ direction outside of a certain conic region.

\begin{thm}[Landweber]\label{thm:Landweber}
Define
\begin{equation}\label{eq:gamma}
\gamma(m) := \#\{k \: : \: 0 < k \le m, \: k \equiv 0,1,2,4 \mod 8\}
\end{equation}
Outside of the region
$$ j-1 \le i \le 2j $$
there are isomorphisms
$$ (\pi^{C_2}_{i,j})_{a-\mr{tors}} \cong (\pi^{C_2}_{i,j+2^{\gamma(i-1)}})_{a-\mr{tors}}. $$
\end{thm}

\begin{proof}
Outside of the region described, the map
$$ \lambda_{i,j} \to (\pi^{C_2}_{i,j})_{a-\mr{tors}} $$
is an isomorphism, and Landweber \cite[Thm.~2.4, Prop.~6.1]{Landweber} observed that James periodicity implies that there is an isomorphism
$$ \lambda_{i,j} \cong \lambda_{i,j+2^{\gamma(i-1)}}. $$
\end{proof}

Dugger-Isaksen prove the following theorem \cite{DuggerIsaksenC2}.\footnote{Belmont-Guillou-Isaksen \cite{BGI} have recently improved this isomorphism theorem to the region $i \ge 2j-4$.}

\begin{thm}[Dugger-Isaksen]\label{thm:DuggerIsaksen}
$C_2$-Betti realization induces an isomorphism
$$ \pi^\RR_{i,j}S^{0,0} \to \pi^{C_2}_{i,j}S^{0,0} $$
for $i \ge 3j-5$.
\end{thm}

\begin{figure}
\centering
\includegraphics[width=0.7\linewidth]{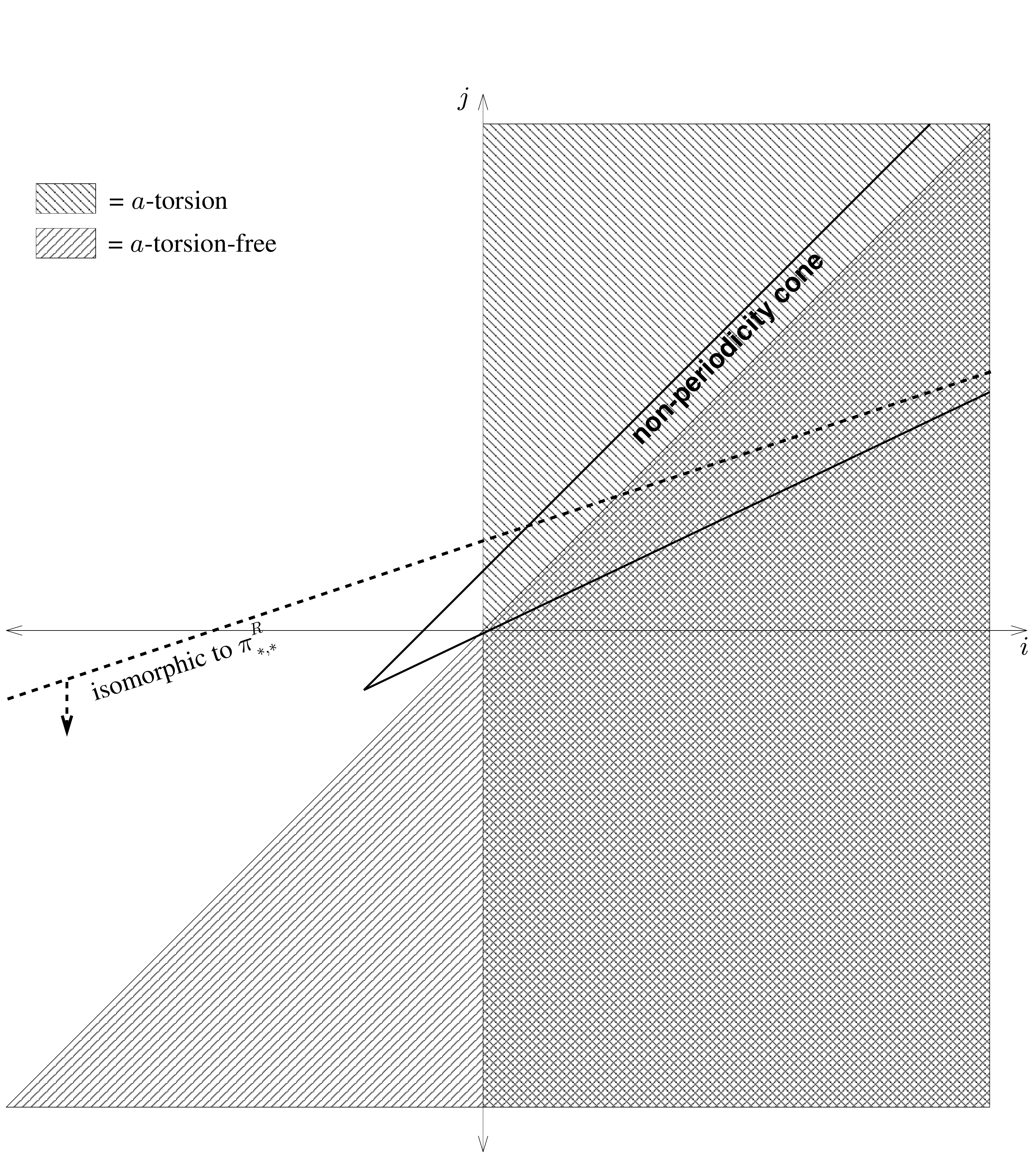}
\caption{The structure of $\pi^{C_2}_{i,j} S^{0,0}$.}
\label{fig:piC2}
\end{figure}

Figure~\ref{fig:piC2} depicts the location of the $a$-torsion and the $a$-towers in $\pi^{C_2}_{*,*}$.  The dashed line marks the region where Dugger-Isaksen proved these groups coincide with the groups $\pi^\RR_{*,*}$ (Theorem~\ref{thm:DuggerIsaksen}).  
This cone in Theorem~{\ref{thm:Landweber}} is labeled the ``non-periodicity cone'' in the figure.  Outside of this cone, the map
$$ \lambda_{i,j} \to (\pi^{C_2}_{i,j})_{a-\mr{tors}} $$
is an isomorphism. 

\subsection*{$\mbf{u}$-self maps}

Since the $C_2$-spectra $S^{1,0}$ and $S^{1,1}$ are non-equivariantly equivalent, the equivalence (\ref{eq:CaC2}) results in a self-equivalence
$$ u: \Sigma^{0,-1} Ca \to Ca. $$
We denote this map $u$, and shall refer to it as a \emph{$u$-self map}, because it induces the multiplication by $u$ map on the homology groups
$$ (H\ul{\FF_2})_{*,*}(Ca) \cong \FF_2[u^{\pm}]. $$

We invite the reader to think of a $u$-self map as analogous to the $v_n$-self maps of chromatic homotopy theory \cite{Ravenel}.  For instance, the mod $2^i$ Moore spectrum admits a $v_1^j$-self map for certain values of $j$ which depend on $i$.  We have the following analog in the present situation.

\begin{thm}\label{thm:uselfmap}
The $C_2$-spectrum $Ca^i$ admits a $u$-self map
$$ u_{2^{\gamma(i-1)}}: \Sigma^{0,-2^{\gamma(i-1)}} Ca^{i} \to Ca^i $$
and this map is an equivalence.
\end{thm}

To prove Theorem~\ref{thm:uselfmap} (and the forthcoming Theorem~\ref{thm:tauselfmap}) we shall need the following lemma.

\begin{lem}\label{lem:ring}
The spectra $C\rho^i \in \SH(\RR)$ and $Ca^i \in \Sp^{C_2}$ are $E_\infty$-ring spectra.
\end{lem}

\begin{proof}
The case of $C\rho^i$ is explained in Remark~\ref{rmk:CrhoEoo}.  The case of $Ca^i$ follows from the fact that $C_2$-Betti realization is monoidal.
\end{proof}



\begin{proof}[Proof of Theorem~\ref{thm:uselfmap}]
Using the equivalence~\ref{eq:Caiequiv} and the Adams isomorphism, we have
\begin{align*} 
\pi_{k,l}^{C_2}Ca^i 
& = [S^{k-l+l\sigma}, \Sigma^{1-i\sigma}S(i\sigma)_+]^{C_2} \\
& \cong [S^{k-l}, \Sigma S(i\sigma)_+ \wedge S^{(-l-i)\sigma}]^{C_2} \\
& \cong [S^{k-l}, \Sigma S(i\sigma)_+ \wedge_{C_2} S^{(-l-i)\sigma}] \\
& \cong \pi_{k-l} \Sigma P_{-l-i}^{-l-1}
\end{align*}
and a similar argument yields 
\begin{equation}\label{eq:HCa^i}
(H\ul{\FF_2})_{k,l}Ca^i \cong (H\FF_2)_{k-l}\Sigma P^{-l-1}_{-l-i}.
\end{equation}
It follows that
$$ (H\ul{\FF_2})_{*,*}Ca^i \cong \FF_2[u^{\pm}, a]/(a^i) $$
where, under the isomorphism (\ref{eq:HCa^i}), the monomial $u^sa^t$ is the homology class coming from the $(s-1)$-cell of $P^{s+t-1}_{s+t-i}$.

By Lemma~\ref{lem:ring}, to prove the theorem it suffices to prove that there is an element
$$ u_{2^{\gamma(i-1)}} \in \pi^{C_2}_{0, -2^{\gamma(i-1)}} Ca^i $$
whose Hurewicz image is
$$ u^{2^{\gamma(i-1)}} \in H^{C_2}_{0,-2^{\gamma(i-1)}}Ca^i. $$
Using the commutative diagram
$$
\xymatrix{
\pi^{C_2}_{0,-2^{\gamma(i-1)}}Ca^i \ar[r]^-{\cong} \ar[d] &
\pi_{2^{\gamma(i-1)}} \Sigma P^{2^{\gamma(i-1)}-1}_{2^{\gamma(i-1)}-i} \ar[d] \\
H^{C_2}_{0,-2^{\gamma(i-1)}} Ca^i \ar[r]_-{\cong} &
H_{2^{\gamma(i-1)}}\Sigma P^{2^{\gamma(i-1)}-1}_{2^{\gamma(i-1)}-i}
}
$$
relating equivariant and non-equivariant Hurewicz homomorphisms, the result follows from the fact (see \cite[Thm.~V.2.14(v)]{Hinfty}) that $P^{2^{\gamma(i-1)}-1}_{2^{\gamma(i-1)}-i}$ is reducible.

The resulting self-map $u_{2^{\gamma(i-1)}}$ induces multiplication by $u^{2^{\gamma(i-1)}}$ on homology, and therefore is a homology isomorphism, and hence is a equivalence.
\end{proof}

Note that we make no claims that these $u$-self maps have any uniqueness or compatibility properties.

\subsection*{$\pmb{\tau}$-self maps}

\begin{thm}\label{thm:tauselfmap}
The $\RR$-motivic spectrum $C\rho^i$ admits a $\tau$-self map
$$ \tau_{2^{\gamma(i-1)}}: \Sigma^{0,-2^{\gamma(i-1)}} C\rho^{i} \to C\rho^i. $$
\end{thm}

\begin{proof}
By Lemma~\ref{lem:ring}, it suffices to prove that there is an element
$$ \tau_{2^{\gamma(i-1)}} \in \pi^{C_2}_{0, -2^{\gamma(i-1)}} C\rho^i $$
whose Hurewicz image is
$$ \tau^{2^{\gamma(i-1)}} \in (H\ul{\FF_2})_{0,-2^{\gamma(i-1)}}C\rho^i \cong \FF_2[\tau, \rho]/(\rho^i). $$
By Theorem~\ref{thm:DuggerIsaksen}, there are isomorphisms in the map of 
long exact sequences
$$
\xymatrix@C-1em{ 
\pi^{\RR}_{i,i-2^{\gamma(i-1)}} \ar[r]_-{\rho^i} \ar[d]_{\cong} & \pi^{\RR}_{0,-2^{\gamma(i-1)}} \ar[r] \ar[d]_\cong & \pi^{\RR}_{0,-2^{\gamma(i-1)}} C\rho^i \ar[r] \ar[d] & \pi^{\RR}_{i-1,i-2^{\gamma(i-1)}} \ar[r]_-{\rho^{i}} \ar[d]_\cong & 
\pi^{\RR}_{-1,-2^{\gamma(i-1)}}\ar[d]_\cong & 
\\
\pi^{C_2}_{i,i-2^{\gamma(i-1)}} \ar[r]^-{a^i}  & \pi^{C_2}_{0,-2^{\gamma(i-1)}} \ar[r] & 
\pi^{C_2}_{0,-2^{\gamma(i-1)}} Ca^i \ar[r] & \pi^{C_2}_{i-1,i-2^{\gamma(i-1)}} \ar[r]^-{a^{i}} & \cdots
\pi^{C_2}_{-1,-2^{\gamma(i-1)}} & 
}
$$
which, by the $5$-lemma, allow us to deduce that there is an isomorphism
$$ \pi^{\RR}_{0,-2^{\gamma(i-1)}} C\rho^i \xto{\cong} 
\pi^{C_2}_{0,-2^{\gamma(i-1)}} Ca^i.$$
The desired element $\tau_{2^{\gamma(i-1)}}$ can be taken to be an element which corresponds, under this isomorphism, to the element $u_{2^{\gamma(i-1)}}$ of Theorem~\ref{thm:uselfmap}.
\end{proof}

\subsection*{$\tau$-self maps at an odd prime}

\emph{In this subsection, everything is implicitly $p$-complete for a fixed odd prime $p$.}

Consider the homotopy complete ($p$-complete) $C_2$-equivariant sphere $S^h$.  We have
\begin{align*}
\pi^{C_2}_{0,k}S^h & = [S^{k\sigma-k}, F((EC_2)_+,S)]^{C_2} \\
& \cong [(EC_2)_+ \wedge S^{k\sigma}, S^{k}]^{C_2} \\
& \cong [(EC_2)_+ \wedge_{C_2} S^{k\sigma}, S^{k}] \\
& \cong [P^\infty_{k}, S^{k}] \\
& \cong 
\begin{cases}
\ZZ_p, & \text{$k$ even}, \\
0, & \text{$k$ odd}
\end{cases}
\end{align*}
where the last isomorphism comes from the fact that $P^\infty_{k}$ is $p$-adically contractible if $k$ is odd, and inclusion of the bottom cell
$$ S^{k} \hookrightarrow P^\infty_{k} $$
is a $p$-adic equivalence if $k$ is even.  Define $u^2$ to be a generator of $\pi^{C_2}_{0,-2}S^h$.  Then the above calculation implies that 
$$ \pi^{C_2}_{0,*}S^h \cong \ZZ_p[u^{\pm 2}]. $$
Thus the homotopy groups of any $p$-complete homotopy complete $C_2$-equivariant spectrum are $u^2$-periodic.

\begin{prop}\label{prop:tauodd}
We have
$$ \pi^\RR_{0,*}S^\wedge_{\rho} \cong \ZZ_p[\tau^2] $$
and every ($p$-complete) $\rho$-complete real motivic spectrum has a $\tau^2$-self map.
Moreover, we have
$$ \widehat\Be{}_p^{C_2}(\tau^2) = u^2. $$
\end{prop}

\begin{proof}
Let $BPGL$ be the odd primary real motivic Brown-Peterson spectrum constructed in \cite{Stahn}. By \cite[Prop.~2.5]{Stahn}, we have
$$ \pi^\RR_{*,*}BPGL = \ZZ_p[\tau^2, v_1, v_2, \cdots ] $$
with $\abs{v_i} = (2p^i-2, p^i-1)$.  
Consider the associated ($p$-complete) real motivic Adams-Novikov spectral sequence\footnote{For convergence, the argument of \cite[Sec.~8]{DuggerIsaksenMASS} shows that this spectral sequence converges to the $(H\FF_p)_\RR$-completion of the motivic sphere spectrum, which by \cite{HuKrizOrmsby}, is the $(p,\eta)$-completion.}
$$ \Ext_{BPGL_{*,*}BPGL}(BPGL_{*,*}, BPGL_{*,*}) \Rightarrow \pi^\RR_{*,*}S^{\wedge}_{\eta}. $$
Stahn \cite{Stahn} explains the odd primary analog of a recipe of Dugger-Isaksen \cite{DuggerIsaksenMASS}, which allows one to completely construct the motivic Adams-Novikov spectral sequence from the classical Adams-Novikov spectral sequence.  In particular, using Proposition~\ref{prop:eta}, we are able to deduce the first statement.  The second statement follows by considering the composite (arising from the Hurewicz homomorphism, the map of \cite{Hoyois}, and Betti realization)
$$ \ZZ_p[\tau^2] = \pi^\RR_{0,*} S^\wedge_{\rho} \rightarrow \pi^\RR_{0,*} BPGL \rightarrow \pi^\RR_{0,*} (H\FF_p)_\RR \rightarrow \pi^{C_2}_{0,*} H\ul{\FF_p} = \FF_p[u^{\pm 2}]. $$
Thm.~4.18 of \cite{HellerOrmsby} implies that $C_2$-Betti realization maps $\tau^2$ to $u^2$.  We deduce that
$$ \widehat\Be{}_p^{C_2}(\tau^2) = \lambda u^2 $$
with $\lambda \in \ZZ_p^\times$.  Without loss of generality, we may choose the generator $\tau^2$ so that $\lambda = 1$. 
\end{proof}

\section{The equivariant-motivic situation}\label{sec:the-equivariant-motivic-situation}

\subsection*{The monoidal Barr-Beck theorem for etale base change}

For a subgroup $H \leq G$, the restriction-induction adjunction
\[ \adjunct{\Res^G_H}{\Sp^G}{\Sp^H}{\Ind^G_H} \]
satisfies the hypotheses of Thm. \ref{thm:monoidalBarrBeck} (c.f. \cite[Thm. 5.32]{MATHEW2017994}). 

Let $\zeta$ denote the map 
$$\zeta: \Spec \CC \to \Spec \RR$$
and consider the induced adjunction
\[ \adjunct{\zeta^\ast}{\SH(\RR)}{\SH(\CC)}{\zeta_\ast}. \]
Note that since the adjunction $\zeta^{\ast} \dashv \zeta_{\ast}$ is monoidal, we have
$$ \Spec \CC_+ = \zeta_* 1 \in \CAlg(\SH(\RR)). $$
With the adjunction $\zeta^{\ast} \dashv \zeta_{\ast}$ being our main situation of interest, we now make the analogous observation in the motivic context.

\begin{prop} \label{prop:EtaleMonoidalBarrBeck} If $f: T \to S$ is finite etale, then $f^{\ast} \dashv f_{\ast}$ satisfies the hypotheses of Thm. \ref{thm:monoidalBarrBeck}, and we have
$$ \SH(T) \simeq \Mod_{\SH(S)}(f_*1). $$
\end{prop}

\begin{proof} In view of the properties of the base change functors outlined in Section~\ref{sec:background}, it only remains to show that $f_{\ast}$ is conservative, so suppose $X \in \SH(T)$ such that $f_{\ast} X \simeq 0$. Consider the pullback square
\[ \begin{tikzcd}[row sep=4ex, column sep=4ex, text height=1.5ex, text depth=0.25ex]
T \times_S T \ar{r}{g} \ar{d}{g} & T \ar{d}{f} \\
T \ar{r}{f} & S.
\end{tikzcd} \]
Then $f^\ast f_\ast x \simeq g_\ast g^\ast x \simeq 0$. But $g_\ast g^\ast x$ is a finite coproduct of copies of $x$, using that $f$ is finite etale. Hence, $x \simeq 0$.
\end{proof}

\begin{cor}\label{cor:EtaleMonoidalBarrBeck}
We have
$$ \SH(\CC) \simeq \Mod_{\SH(\RR)}(\Spec \CC_+). $$
\end{cor}

The following is the key calculational observation behind this paper.

\begin{prop} \label{prop:cofiberRho} There is a non-canonical map
	$$ C(\rho) \to \Spec \CC_+ $$
which becomes an equivalence after $p$-completion and cellularization.	
\end{prop}

\begin{proof} Let 
$$\xi: S^{0,0}_{\RR} \to \Spec \CC_+$$ 
be the unit map, which is adjoint to the identity in $\SH(\CC)$. By adjunction, we have
\[  [S^{-1,-1}, \Spec \CC_+]_\RR \cong [S^{-1,-1}, S^{0,0}]_\CC. \]
But since $\rho \simeq 0$ in $\SH(\CC)$, $\xi \circ \rho$ is null homotopic. Making a choice of null homotopy, we obtain a comparison map $$\alpha: C(\rho) \to \Spec \CC_+$$ 
that we wish to show is a $p$-complete cellular equivalence. Using the motivic Adams spectral sequence, it suffices to show that
\[ \beta: (H\FF_p)^{\RR}_{\ast,\ast}(C(\rho))  \to (H\FF_p)^{\RR}_{\ast,\ast}(\Spec \CC_+) \cong \pi^{\CC}_{\ast,\ast}(H\FF_p)_{\CC} \]
is an isomorphism (for $p$ odd, the motivic Adams spectral sequence only converges to the $(p,\eta)$-completion \cite{HellerOrmsby}, but both $C(\rho)$ and $\Spec \CC_+$ are $\eta$-complete by Prop.~\ref{prop:eta}). 

For $p = 2$, $\pi^{\RR}_{\ast,\ast}$ of the map $(H\FF_2)_\RR \to \zeta_{\ast} (H\FF_2)_{\CC}$ is computed to be the surjection $\FF_2[\tau,\rho] \to \FF_2[\tau]$, which identifies $\beta$ as the isomorphism $\FF_2[\tau,\rho]/\rho \cong \FF_2[\tau]$.

For $p$ odd, $\pi^{\RR}_{\ast,\ast}$ of the map $(H\FF_p)_{\RR} \to \zeta_{\ast} (H\FF_p)_{\CC}$ is computed (see \cite[Prop.~1.1]{Stahn}) to be the injection $\FF_p[\tau^2] \to \FF_p[\tau]$.  Using the fact that $\rho$ acts trivially, we deduce
$$ (H\FF_p)^\RR_{*,*}C\rho \cong \FF_p[\tau^2]\{1,\tau\} $$
and we conclude that $\beta$ is an isomorphism.
\end{proof}

\begin{rmk}\label{rmk:cofiberRho}
We claim that $\Spec \CC_+$ is not cellular in $\SH(\RR)$. Indeed, upon applying $\zeta^\ast$, the cofiber sequence
\[ S^{-1,-1} \xto{\rho} S^{0,0} \to C(\rho) \]
becomes
\[ S^{-1,-1} \xto{0} S^{0,0} \to \zeta^*(C (\rho)) \]
and thus we have
\[\zeta^*(C\rho) \simeq S^{0,0} \vee S^{0,-1}. \]
But
$$\zeta^\ast \Spec \CC_+ =\zeta^\ast \zeta_\ast 1 = S^{0,0} \vee S^{0,0}.$$ 
In effect, the presence of the motivic weight forbids $\Spec \CC_+$ from being cellular.
\end{rmk}

\begin{rmk}\label{rmk:CrhoEoo} Via Prop. \ref{prop:cofiberRho} and $\Cell$ being lax monoidal, $C(\rho)^\wedge_p$ and therefore $C(\rho^n)^\wedge_p$ obtain the structure of $E_{\infty}$-algebras in $\SH(\RR)^{\wedge}_p$.
\end{rmk}

\begin{cor}\label{cor:CellularBarrBeck} There is an equivalence
$$\SH_{\cell}(\CC)^{\wedge}_p \simeq \Mod_{\SH_{\cell}(\RR)^{\wedge}_p}( C(\rho) )$$ 
and we have a diagram of commuting left adjoints
\[ \begin{tikzcd}[row sep=4ex, column sep=4ex, text height=1.5ex, text depth=0.25ex]
\SH_{\cell}(\RR)^{\wedge}_p \ar{d}{\zeta^{\ast}} \ar[hookrightarrow]{r} & \SH(\RR)^{\wedge}_p \ar{d}{\zeta^{\ast}} \\
\SH_{\cell}(\CC)^{\wedge}_p \ar{d}{\simeq}  \ar[hookrightarrow]{r} & \SH(\CC)^{\wedge}_p \ar{d}{\simeq}  \\
\Mod_{\SH_{\cell}(\RR)^{\wedge}_p}( C(\rho) ) \ar[hookrightarrow]{r} & \Mod_{\SH(\RR)^{\wedge}_p}(\Spec \CC_+) \\
\end{tikzcd} \]
where the horizontal right adjoints are given by the cellularization functor.
In particular, for $X \in \Sp(\RR)^{\wedge}_p$, we have an induced isomorphism
$$ \pi^{\RR}_{*,*}X \wedge C(\rho) \cong \pi^\CC_{*,*}\zeta^*X. $$
\end{cor}
\begin{proof} Combine Prop.~\ref{prop:cofiberRho}, Prop.~\ref{prop:EtaleMonoidalBarrBeck}, Lem.~\ref{lem:CellularizationOnModules} with $A = \Spec(\CC)_+$, and Lem.~\ref{lm:BousfieldLocalizationOfCellular} for the $p$-completion.
\end{proof}

\begin{wrn} $\Cell$ is not strong monoidal, and indeed one may show that $$\Cell(\Spec \CC_+ \wedge \Spec \CC_+) \not\simeq C(\rho)^{\wedge 2}.$$ Therefore, we don't have an induced adjunction between $\Spec \CC_+$-local objects in $\SH(\RR)^{\wedge}_p$ and $C(\rho)$-local objects in $\SH_{\cell}(\RR)^{\wedge}_p$.
\end{wrn}

\subsection*{Betti realization}

We next relate the motivic to the $C_2$-equivariant situation. We begin by recalling the Betti realization and constant functors, for which an $\infty$-categorical reference is \cite[\S 10.2, 11]{BachmannHoyoisNorms}.

\begin{dfn} The complex Betti realization functor 
$$\Be: \SH(\CC) \to \Sp$$ 
is the unique colimit preserving functor that sends the complex motivic spectrum $\Sigma^{\infty}_+ X$ to $\Sigma^{\infty}_+ X(\CC)$ for $X$ a smooth quasi-projective $\CC$-variety, where $X(\CC)$ is endowed with the analytic topology. Likewise, the $C_2$-Betti realization functor $$\Be^{C_2}: \SH(\RR) \to \Sp^{C_2}$$ 
is the unique colimit preserving functor that sends the real motivic spectrum $\Sigma^{\infty}_+ X$ to $\Sigma^{\infty}_+ X(\CC)$ for $X$ a smooth quasi-projective $\RR$-variety, where $X(\CC)$ has $C_2$-action given by complex conjugation. 
We define $p$-complete Betti realization functors by
\begin{align*}
\widehat\Be_p(-) & := \Be(-)^{\wedge}_p, \\
\widehat\Be{}_p^{C_2}(-) & := \Be^{C_2}(-)^{\wedge}_p.
\end{align*}
\end{dfn}

Both $\Be$ and $\Be^{C_2}$ are symmetric monoidal functors. Let $\Sing$ and $\Sing^{C_2}$ denote their respective right adjoints, so we have the following diagram of adjoint functors:
\begin{equation}
\xymatrix@R+1em@C+1em{
\SH({\RR}) \ar@<.5ex>[r]^{\Be^{C_2}} \ar@<-.5ex>[d]_{\zeta^*}   & \Sp^{C_2} \ar@<.5ex>[l]^{\Sing^{C_2}} \ar@<-.5ex>[d]_{\Res_e^{C_2}} \\
\SH({\CC}) \ar@<.5ex>[r]^{\Be} \ar@<-.5ex>[u]_{\zeta_*} & \Sp \ar@<-.5ex>[u]_{\Ind_e^{C_2}}
\ar@<.5ex>[l]^{\Sing} 
}
\end{equation}

We also have the real Betti realization functor 
$$ \Be_{\RR}: \SH(\RR) \to \Sp$$ 
that sends $\Sigma^{\infty} X_+$ to $\Sigma^{\infty}_+ X(\RR)$. By definition, $\Phi^{C_2} \Be^{C_2} \simeq \Be_{\RR}$. Bachmann has also identified his real-etale localization functor $$\SH(\RR) \to \Sp$$ with $\Be_{\RR}$ (\cite[\S 10]{Bachmann}). If we let 
$$ i_{\ast} : \Sp \to \Sp^{C_2} $$ 
denote the right adjoint to geometric fixed points $(-)^{\Phi{C_2}}$, then it follows that $\Sing^{C_2} i_{\ast}$ is fully faithful.

Consider the $\rho$-inverted motivic sphere $S^{0,0}[\rho^{-1}]$ and the associated localization $\SH(S)[\rho^{-1}]$. 
The following main theorem of \cite{Bachmann} is essential.

\begin{thm}[\cite{Bachmann}]\label{thm:Bachmann}
There is an equivalence of $\infty$-categories
$$\SH(S)[\rho^{-1}] \simeq \Sp(\Shv(\Sper(S))),$$
where $\Sper(S)$ is the real spectrum of $S$ (\cite[\S 3]{Bachmann}). In particular, we have 
$$ \SH(\RR)[\rho^{-1}] \simeq \Sp $$ 
and the following diagram commutes
$$
\xymatrix{
\SH(\RR) \ar[rr]^{\Be_\RR} \ar[dr] && \Sp \\
& \SH(\RR)[\rho^{-1}] \ar[ur]_\simeq  
}
$$
Thus, real Betti realization is localization with respect to $\rho$.
\end{thm}

We recall the definition of the constant functor, and Heller-Ormsby's equivariant generalization \cite{HellerOrmsby}.

\begin{dfn}\label{def:HellerOrmsby} The constant functor 
$$c^*_\CC: \Sp \to \SH(\CC)$$ 
is the unique colimit preserving functor that sends $S^0$ to $S^{0,0}$. The $C_2$-equivariant constant functor 
$$c^*_\RR: \Sp^{C_2} \to \SH(\RR)$$ 
is the unique colimit preserving functor that sends $S^0 = {C_2/C_2}_+$ to $S^{0,0} = \Spec \RR_+$ and $C_2/1_+$ to $\Spec \CC_+$.  
\end{dfn}

\begin{lem} \label{lm:BettiSplitsConstant} Betti realization splits the constant functor. In other words, we have equivalences
\begin{gather*} 
\Be \circ c^*_{\CC} \simeq \id, \\
\quad \Be^{C_2} \circ c^*_{\RR} \simeq \id.
\end{gather*}
\end{lem}
\begin{proof} The functors in question preserve colimits, so it suffices to observe that: 
\begin{align*}
(\Be c^*_{\CC})(S^0) & = S^0, \\
(\Be^{C_2} c^*_{\RR})(S^0) & = S^0, \\ 
\text{and}\quad (\Be^{C_2} c^*_{\CC})(C_2/1_+) & = C_2/1_+.
\end{align*}
\end{proof}

\begin{lem} \label{lem:EquivariantSpectraIsModuleCategoryOverMotivicSpectra} The monoidal adjunctions
\begin{gather*}
\adjunct{\Be}{\SH(\CC)}{\Sp}{\Sing} \\
\adjunct{\Be^{C_2}}{\SH(\RR)}{\Sp^{C_2}}{\Sing^{C_2}}
\end{gather*}
satisfy the hypotheses of Thm.~\ref{thm:monoidalBarrBeck}. Therefore, we have
\begin{gather*}
\Sp \simeq \Mod_{\SH(\CC)}(\Sing S^0), \\
\Sp^{C_2} \simeq \Mod_{\SH(\RR)}(\Sing^{C_2} S^{0}). 
\end{gather*}
\end{lem}
\begin{proof} We verify the second statement; the first will follow by a similar argument. Let us consider the hypotheses in turn:
\begin{enumerate}
\item In view of Lem.~\ref{lm:BettiSplitsConstant}, $\Sing^{C_2}$ is conservative as it is split by the right adjoint to the constant functor $c^*_{\RR}$.
\vspace{10pt}

\item Note that for $X$ a smooth quasi-projective $\RR$-variety, $X(\RR)$ and $X(\CC)$ have the homotopy types of finite CW-complexes, hence $\Be^{C_2}(\Sigma^{\infty}_+ X)$ is compact in $\Sp^{C_2}$. Because the collection of motivic spectra $\{ \Sigma^{\infty}_+ X \}$ furnish a set of compact generators for $\SH(\RR)$, we deduce that $\Sing^{C_2}$ preserves colimits. To verify the projection formula 
$$\Sing^{C_2}(A) \wedge B \simeq \Sing^{C_2}(A \wedge \Be^{C_2} B),$$ because both sides preserve colimits in the $B$ variable, it suffices to check for $B = \Sigma^{\infty}_+ X$. In this case, we need to show that for any $W \in \SH(\RR)$, the comparison map
\[ [W,\Sing^{C_2}(A) \wedge B]_{\RR} \to [W,\Sing^{C_2}(A \wedge \Be^{C_2} B)]_{\RR} \]
is an isomorphism. Using that $B$ is dualizable, under adjunction this is equivalent to
\[ [\Be^{C_2} (W) \wedge \Be^{C_2}(B ^{\vee}),A]{}^{C_2} \to [\Be^{C_2} (W),A \wedge \Be^{C_2} B]{}^{C_2}  \]
where the conclusion follows because $\Be^{C_2}B$ is also dualizable with dual given by $\Be^{C_2}(B^{\vee})$.
\end{enumerate}
\end{proof}

Using Lem.~\ref{lem:PassingLocalizationToModuleCat}, we deduce the following $p$-complete variant.

\begin{cor}
For a prime $p$, we have
\begin{gather*}
 \Sp^{\wedge}_p \simeq \Mod_{\SH(\CC)^{\wedge}_p}([\Sing S^0]^{\wedge}_p), \\
[\Sp^{C_2}]^{\wedge}_p \simeq \Mod_{\SH(\RR)^{\wedge}_p}([\Sing^{C_2} S^{0}]^{\wedge}_p).
\end{gather*}
\end{cor}

We may also deduce the following cellular variant, which highlights an important difference between the $\RR$- and $\CC$-motivic settings.

\begin{cor}
The adjunction
$$ \adjunct{\Be}{\SH_{\cell}(\CC)}{\Sp}{\Cell \Sing} $$
satisfies the hypotheses of Theorem~\ref{thm:monoidalBarrBeck}, and therefore Betti realization gives an equivalence
$$ \Mod_{\SH_{\cell}(\CC)}(\Cell \Sing S^0) \simeq \Sp. $$
In particular, we have an equivalence
$$ \Mod_{\SH_{\cell}(\CC)}(\Cell \Sing S^0) \simeq \Mod_{\SH(\CC)}(\Sing S^0). $$
In the real case, the adjunction
$$
\adjunct{\Be^{C_2}}{\SH_{\cell}(\RR)}{\Sp^{C_2}}{\Cell \Sing^{C_2}},\\  
$$
satisfies these hypotheses after $p$-completion, giving 
\begin{equation}\label{eq:equivbettibarbeck}
 \Mod_{\SH_{\cell}(\RR)^{\wedge}_p}([\Cell \Sing^{C_2} S^0]^{\wedge}_p) \simeq (\Sp^{C_2})^{\wedge}_p.
 \end{equation}
\end{cor}

\begin{proof}
Lem.~\ref{lm:CellularProjectionFormula} implies every hypotheses of Thm.~\ref{thm:monoidalBarrBeck} holds for the cellular adjuctions except for the conservativity hypothesis.  In the complex case, because $c^*_{\CC}$ has essential image in $\SH_{\cell}(\CC)$, $\Cell \Sing$ is conservative. 
However, in the real case,  
$$c^*_{\RR}(C_2/1_+) = \Spec \CC_+$$ 
is not cellular. Nonetheless, because 
$$ (\Cell \Spec \CC_+)^{\wedge}_p \simeq C(\rho)^{\wedge}_p \to (\Spec \CC_+)^{\wedge}_p$$ 
is sent to an equivalence in $(\Sp^{C_2})^{\wedge}_p$, it follows that $\Cell \Sing^{C_2}$ is conservative after $p$-completion. 
\end{proof}

\begin{rmk}
The observation (\ref{eq:equivbettibarbeck}) is not new --- Ricka proves proves this in \cite[Thm.~2.4]{Ricka}.  However, Ricka's version does not have the $p$-completion.  We believe the subtlety metioned in the proof above may have been overlooked in his proof, however, and we  
do not know if (\ref{eq:equivbettibarbeck}) holds without the $p$-completion.  
\end{rmk}

\subsection*{Betti realization as a localization}

We will now show that in the $p$-complete setting, both $\Cell \Sing$ and $\Cell \Sing^{C_2}$ are fully faithful, implying $\widehat{\Be}_p$ and $\widehat\Be{}_p^{C_2}$ are localizations when restricted to $p$-complete cellular motivic spectra.

The complex case, summarized in the following theorem, was essentially proven by Dugger and Isaksen \cite{DuggerIsaksenMASS} (in the case of $p = 2$) and Stahn \cite{Stahn} (in the case of $p$ odd).

\begin{thm} \label{thm:ComplexSingFullyFaithful} 
The functor
$$\Cell \Sing: \Sp^{\wedge}_p \to \SH_{\cell}(\CC)^{\wedge}_p$$ is fully faithful with essential image consisting of those objects in $\SH_{\cell}(\CC)^{\wedge}_p$ on which multiplication by $\tau$ is an equivalence.  Therefore, given $X \in \SH_{\cell}(\CC)^{\wedge}_p$, $2$-complete Betti realization induces an isomorphism
$$ \pi^\CC_{i,j} X[\tau^{-1}] \xto{\cong} \pi_i \widehat{\Be}_p(X). $$ 
\end{thm}

\begin{proof} Because we already know that $\Be \dashv \Cell \Sing$ satisfies the hypotheses of Thm.~\ref{thm:monoidalBarrBeck}, it suffices to compute $(S^{0,0})^{\wedge}_p[\tau^{-1}] \simeq \Cell \Sing(S^0)^{\wedge}_p$. But the natural map $$(S^{0,0})^{\wedge}_p[\tau^{-1}] \to \Sing(S^0)^{\wedge}_p$$ is a cellular equivalence by the results of \cite{DuggerIsaksenMASS} and \cite{Stahn}.
\end{proof}

Our strategy will be to formally derive the real case from this, by lifting this localization up the $\rho$-completion tower, and combining with Bachmann's Thm.~\ref{thm:Bachmann}.

To this end, we consider the 
isotropy separation recollement on $\Sp^{C_2}$ given by
\[ \begin{tikzcd}[row sep=4ex, column sep=8ex, text height=1.5ex, text depth=0.25ex]
\Sp^{h C_2} \ar[shift right=1]{r}[swap]{j_*} & \Sp^{C_2} \ar[shift left=1]{r}{(-)^{\Phi C_2}} \ar[shift right=1]{l}[swap]{(-)^h} & \Sp \ar[shift left=1]{l}{i_*}.
\end{tikzcd} \]

\begin{lem} \label{lm:VanishingOfWrongWayGluingFunctor} We have equivalences of functors
\begin{align*}
 (\Be^{C_2} \Sing^{C_2} i_*(-))^{h} & \simeq 0, \\
 (\Be^{C_2} \Cell \Sing^{C_2} i_*(-))^h & \simeq 0.
 \end{align*}
\end{lem}
\begin{proof} Because $S^{0,0}[\rho^{-1}]$ is cellular, the essential image of $$\Sing^{C_2} i_*: \Sp \to \SH(\RR)$$ is cellular as it is generated as a localizing subcategory by $S^{0,0}[\rho^{-1}]$. Therefore, $$\Cell \Sing^{C_2} i_* \simeq \Sing^{C_2} i_*,$$ so we may ignore cellularization in the proof. Because for $E \in \Sp^{C_2}$, $E^h \simeq 0$ if and only if $\Res^{C_2}_e E \simeq 0$, it suffices to show that $$\Res^{C_2}_e \Be^{C_2} \Sing^{C_2} i_* \simeq 0.$$ Because $\Res^{C_2}_e \Be^{C_2} \simeq \Be \zeta^{\ast}$ for $$\zeta: \Spec \CC \to \Spec \RR,$$ this follows from the observation that $$\zeta^{\ast}: \SH(\RR) \to \SH(\CC)$$ vanishes on $\rho$-inverted objects. 
\end{proof}

\begin{lem} \label{lm:EquivalenceOfRightWayGluingFunctors} The natural transformations
\begin{align*}
(\Be^{C_2} \Sing^{C_2} j_{\ast}(-))^{\Phi C_2} & \to  (j_{\ast}(-))^{\Phi C_2}, \\
(\Be^{C_2} \Cell \Sing^{C_2} j_{\ast}(-))^{\Phi C_2} & \to (j_{\ast}(-))^{\Phi C_2}
\end{align*}
induced by the counits of the adjunctions 
\begin{gather*}
\epsilon: \Be^{C_2} \Sing^{C_2} \to \id,  \: \mr{and} \\
\epsilon': \Cell \Be^{C_2} \Sing^{C_2} \to \id
\end{gather*}
 are equivalences.
\end{lem}

\begin{proof} We first consider the non-cellular assertion. Let $X \in \Sp^{h C_2}$ and $Y = j_{\ast} X$. Since $i_{\ast}$ is fully faithful, it suffices to prove that
\[ i_{\ast} \left( [ \Be^{C_2} \Sing^{C_2} Y]^{\Phi C_2} \right) = \widetilde{EC_2} \wedge \Be^{C_2} \Sing^{C_2} Y \to i_{\ast} (Y^{\Phi C_2}) = \widetilde{EC_2} \wedge Y \] is an equivalence. For this, first note that because $\widetilde{EC_2} = \Be^{C_2}(S^{0,0}[\rho^{-1}])$, using that $\Be^{C_2}$ is strong monoidal and the projection formula we have equivalences
\begin{align*} \widetilde{EC_2} \wedge \Be^{C_2} \Sing^{C_2} j_{\ast} (X) & \simeq \Be^{C_2} (\Sing^{C_2} (X) \wedge S^{0,0}[\rho^{-1}]) \\
 & \simeq \Be^{C_2} \Sing^{C_2} (X \wedge \widetilde{EC_2})
\end{align*}
under which $\widetilde{EC_2} \wedge \epsilon_Y$ is identified with $\epsilon_{Y \wedge \widetilde{EC_2}}$. Next, by 
Lem.~\ref{lm:VanishingOfWrongWayGluingFunctor} and the fact that $\Sing^{C_2} i_{\ast}$ is fully faithful, for any $Z \in \Sp$ the fiber sequence of functors
\[ (EC_2)_+ \wedge- \to \id \to \td{EC}_2 \wedge - \]
applied to $\Be^{C_2} \Sing^{C_2} i_{\ast} Z$ yields the equivalence
\[  \Be^{C_2} \Sing^{C_2} i_{\ast} Z \to \td{EC}_2 \wedge \Be^{C_2} \Sing^{C_2} i_{\ast} Z \simeq i_{\ast} Z. \]
In particular, the counit $$\Be^{C_2} \Sing^{C_2} (X \wedge \widetilde{EC_2}) \to X \wedge \widetilde{EC_2}$$ is an equivalence.

Finally, the cellular assertion is proven in the same way, using now that $S^{0,0}[\rho^{-1}]$ is cellular and $\Be^{C_2} \dashv \Cell \Sing^{C_2}$ is a monoidal adjunction that satisfies the projection formula by Lem.~\ref{lm:CellularProjectionFormula}.
\end{proof}

We have almost assembled all of the ingredients needed to prove Thm. \ref{thm:C2-SingFullyFaithful}. In view of 
Lem.~\ref{lem:RecollementFullyFaithful}, it only remains to prove the full faithfulness of $\Sing^{C_2}$ on the Borel part of the recollement, which we turn to now.

Because we have
$\Be^{C_2}(C(\rho)) = C(a)$ for the Euler class 
$$a: S^{-\sigma} \to S^0$$ 
and $(\Sp^{C_2})^{\wedge}_a \simeq \Sp^{h C_2}$ (\ref{eq:acompletion}), we obtain the induced adjunction
\[ \adjunct{\widehat\Be {}_p^{h C_2}}{\SH_{\cell}(\RR)^{\wedge}_{p,\rho}}{(\Sp^{h C_2})^{\wedge}_p}{\Cell \Sing^{h C_2}} \]
as in Lem. \ref{lem:PassingLocalizationToModuleCat}.

\begin{cor} \label{cor:BorelFullyFaithful} The functor $\Cell \Sing^{h C_2}$ is fully faithful.
\end{cor}
\begin{proof} Combine Thm.~\ref{thm:ComplexSingFullyFaithful}, Cor. \ref{cor:CellularBarrBeck}, and Prop. \ref{prop:FullyFaithfulLift}.
\end{proof}

We may now deduce the categorical half of our main theorem, which states that $C_2$-equivariant Betti realization, when restricted to $p$-complete cellular real motivic spectra, is a localization.

\begin{thm} \label{thm:C2-SingFullyFaithful} $\Cell \Sing^{C_2}: (\Sp^{C_2})^{\wedge}_p \to \SH_{\cell}(\RR)^{\wedge}_p$ is fully faithful.
\end{thm}

\begin{proof} The conditions of Lem.~\ref{lem:RecollementFullyFaithful} apply in view of Lem.~\ref{lm:VanishingOfWrongWayGluingFunctor}, 
Lem.~\ref{lm:EquivalenceOfRightWayGluingFunctors}, Bachmann's theorem~\ref{thm:Bachmann}, and Cor.~\ref{cor:BorelFullyFaithful}.
\end{proof}

\subsection*{Computing Betti localization}

In the complex case, Theorem~\ref{thm:ComplexSingFullyFaithful} implies that Betti realization can be computed on $p$-complete cellular complex motivic spectra by inverting $\tau \in \pi^{\CC}_{0,-1}(S^{0,0})^{\wedge}_p$.  

We would like a similar result for the $C_2$-Betti realization of a $p$-complete cellular real motivic spectrum.
In the real case, for $X \in \SH(\RR)$, the isotropy separation recollement implies that the homotopy type of the $p$-complete $C_2$-equivariant Betti realization can then be recovered by the pullback:
\begin{equation}\label{eq:isotropy}
\xymatrix{
\widehat\Be{}_p^{C_2} (X) \ar[r] \ar[d] & \widehat\Be{}^{C_2}_p (X)^{\Phi} \ar[d] \\
\widehat\Be{}_p^{C_2} (X)^h \ar[r] & \widehat\Be{}_p^{C_2} (X)^t
}
\end{equation}
Therefore, it suffices to compute $\widehat\Be{}_p^{C_2}(X)^\Phi$, $\widehat\Be{}_p^{C_2}(X)^h$, $\widehat\Be{}_p^{C_2}(X)^t$, and the map 
$$ \widehat\Be{}_p^{C_2}(X)^\Phi \to \widehat\Be{}_p^{C_2}(X)^t. $$

For the geometric localization $\widehat\Be{}_p^{C_2}(X)^\Phi$, Bachmann's theorem~\ref{thm:Bachmann} has the following immediate consequence (which does not require $p$-completion or cellularization).

\begin{thm}\label{thm:Bachmann2}
For $X \in \SH(\RR)$, equivariant Betti realization induces an isomorphism 
$$ \pi^\RR_{*,*}X[\rho^{-1}] \xto{\cong} \pi^{C_2}_{*,*}\Be^{C_2}(X)^\Phi. $$
\end{thm}

\begin{proof}
Using Bachmann's theorem~\ref{thm:Bachmann}, we have
\begin{align*}
\pi^{\RR}_{i,j} X[\rho^{-1}]
& \cong \pi_{i-j}\Be_\RR(X) \\
& \cong \pi_{i-j}\Be^{C_2}(X)^{\Phi C_2} \\
& \cong \pi^{C_2}_{i,j}\Be^{C_2}(X)^{\Phi}.
\end{align*}
\end{proof}

We will now show that if $X$ is $p$-complete and cellular, the $p$-complete homotopy completion $\widehat\Be{}_p^{C_2}(X)^h$ can be computed by inverting $\tau$ on the $\rho$-completion tower.  The Tate spectrum $\widehat\Be{}_p^{C_2}(X)^t$ may then computed by inverting $\rho$ on the $\tau$-inverted $\rho$-completion.

Let us now describe in detail how to invert $\tau$ on the $\rho$-completion tower.
For every $n$, we have adjunctions
\[ \adjunct{\widehat\Be{}_p^{C_2,n}}{\Mod_{\SH_{\cell}(\RR)^{\wedge}_p}(C(\rho^n))}{\Mod_{(\Sp^{C_2})^{\wedge}_p}(C(a^n))}{\Sing^{C_2,n}} \]
where $\Sing^{C_2,n}$ is fully faithful by Thm.~\ref{thm:ComplexSingFullyFaithful} and Prop. \ref{prop:FullyFaithfulLift}. The self-map $\tau_N$ of $C(\rho^n)^\wedge_p$ constructed in Section~\ref{sec:tau} (where we take $\tau_N := \tau^2$ there in the case of $p$ odd) allows us to explicitly compute the resulting localization functor in terms of $\tau_N$-localization, as stated in the next lemma.  

\begin{lem}\label{lem:Crhonloc}
For $X \in \Mod_{\SH_{\cell}(\RR)^{\wedge}_p}(C(\rho^n))$, we have 
$$ \Sing^{C_2,n} \widehat\Be{}_p^{C_2,n} X \simeq X[\tau_N^{-1}]. $$
Thus, the image of the fully faithful right adjoint
$$ \Sing^{C_2,n} : \Mod_{(\Sp^{C_2})^{\wedge}_p}(C(a^n)) \to \Mod_{\SH_{\cell}(\RR)^{\wedge}_p}(C(\rho^n)) $$
consists of those $p$-complete cellular $C(\rho^n)$-modules on which multiplication by $\tau_N$ is an equivalence.
\end{lem}

\begin{proof}
For brevity, we implicitly assume everything is $p$-complete in this proof.
We claim the self-maps $\tau_N$ satisfy
\begin{enumerate}
    \item $\widehat\Be{}_p^{C_2,n}(\tau_N) = u_N$ is an self-equivalence of $C(a^n)$,
    \vspace{10pt}
    \item $C(\rho^n)[\tau_N^{-1}] \wedge_{C(\rho^n)} C(\rho) \simeq C(\rho)[\tau^{-1}]. $
\end{enumerate}
Statement (1) is proven in Thm.~\ref{thm:uselfmap} (for $p=2$) and Prop.~\ref{prop:tauodd} (for $p$ odd).  For statement (2), it suffices to show that the composite
$$ \Sigma^{0,-N}C(\rho) \simeq \Sigma^{0,-N} C(\rho^n) \wedge_{C(\rho^n)} C(\rho) \xto{\tau_N \wedge 1} C(\rho^n) \wedge_{C(\rho^n)} C(\rho) \simeq C(\rho) $$
is equal to $\tau^N$, up to multiplication by a unit.  However, by Cor.~\ref{cor:CellularBarrBeck}, we have
$$ \pi^{\RR}_{0,*}C(\rho) \cong \pi^\CC_{0,*} (S^{0,0})^{\wedge}_p \cong \ZZ_p[\tau]. $$
In particular, the Hurewicz homomorphism
$$ \pi^{\RR}_{0,*}C(\rho) \to (H\FF_p)^\RR_{0,*}C\rho \cong \FF_p[\tau] $$
is given by the obvious surjection, and the result follows from the fact that $\tau_N$ induces multiplication by $\tau^N$ on homology.

By (1), we have a comparison map 
$$C(\rho^n)[\tau_N^{-1}] \to \Sing^{C_2,n}(C(a^n))$$ 
adjoint to the equivalence $$C(a^n)[u_N^{-1}] \simeq C(a^n).$$ 
After base-change to $C(\rho)$, this map is an equivalence by (2), hence is an equivalence as $- \wedge_{C(\rho^n)} C(\rho)$ is conservative. Because the adjunctions in question also satisfy the hypotheses of Thm. \ref{thm:monoidalBarrBeck}, we have that
\begin{align*}
\Sing^{C_2,n} \widehat\Be{}_p^{C_2,n} X 
& \simeq \Sing^{C_2,n} ((\widehat\Be{}_p^{C_2,n} X) \wedge_{C(a^n)} C(a^n)) \\
& \simeq \Sing^{C_2,n}(C(a^n)) \wedge_{C(\rho^n)} X \\
& \simeq X[\tau_N^{-1}].
\end{align*}
\end{proof}

For $p$ odd, every $X \in \SH(\RR)^{\wedge}_p$ has a $\tau^2$ self-map on its $\rho$-completion, and we can therefore form the telescope
$$ X^\wedge_\rho[\tau^{-1}] := X^\wedge_\rho[\tau^{-2}]. $$
For $p=2$, because the periodicity of the elements $\tau_N$ increases as $n \to \infty$, we do not have an analogous construction.
Nevertheless, given $X \in \SH_{\cell}(\RR)^{\wedge}_p$, the equivalences of Lemma~\ref{lem:Crhonloc} allow us to define maps
\begin{align*}
X \wedge C(\rho^n)[\tau_N^{-1}] & \simeq \Sing^{C_2,n} \widehat\Be{}_p^{C_2,n} X \wedge C(\rho^n) \\
& \to \Sing^{C_2,n-1} \widehat\Be{}_p^{C_2,n-1} X \wedge C(\rho^{n-1}) \\
& \simeq X \wedge C(\rho^{n-1})[\tau^{-1}_{N'}].
\end{align*}
We may therefore define
$$ X^{\wedge}_\rho[\tau^{-1}] := \lim_n X \wedge C(\rho^n)[\tau_N^{-1}]. $$

We are now ready to deduce the computational half of our main theorem.

\begin{thm}\label{thm:main}
For $X \in \SH_{\cell}(\RR)^{\wedge}_p$, we have
\[ \Cell \Sing^{h C_2} \widehat\Be{}_p^{h C_2} X^{\wedge}_{\rho} \simeq X^{\wedge}_\rho [\tau^{-1}]. \]
and $C_2$-Betti realization induces an isomorphism
$$ \pi^\RR_{*,*} X^{\wedge}_\rho[\tau^{-1}] \xto{\cong} \pi^{C_2}_{*,*} \widehat\Be{}_p^{C_2}(X)^h. $$
\end{thm}

\begin{proof}
Since
$$ \Cell \Sing^{h C_2} \widehat\Be{}_p^{h C_2} X^{\wedge}_{\rho} \simeq \lim_n 
\Sing^{C_2,n} \widehat\Be{}_p^{C_2,n} X \wedge C(\rho^n), $$
we deduce the first statement from Lemma~\ref{lem:Crhonloc}.
The second statement follows from the adjunction
\begin{align*}
\pi^{C_2}_{i,j}(\widehat\Be{}_p^{C_2}(X)^h) 
& = [\Be^{C_2}S^{i,j}, \widehat\Be{}_p^{C_2}(X)^h]^{C_2} \\
& \cong [S^{i,j}, \Cell \Sing^{C_2} \widehat\Be{}_p^{C_2}(X)^h]_\RR \\
& \cong \pi^\RR_{i,j}X^{\wedge}_\rho[\tau^{-1}].
\end{align*}
\end{proof}

\section{Examples}\label{sec:examples}

\begin{figure}
\centering
\includegraphics[width=\linewidth]{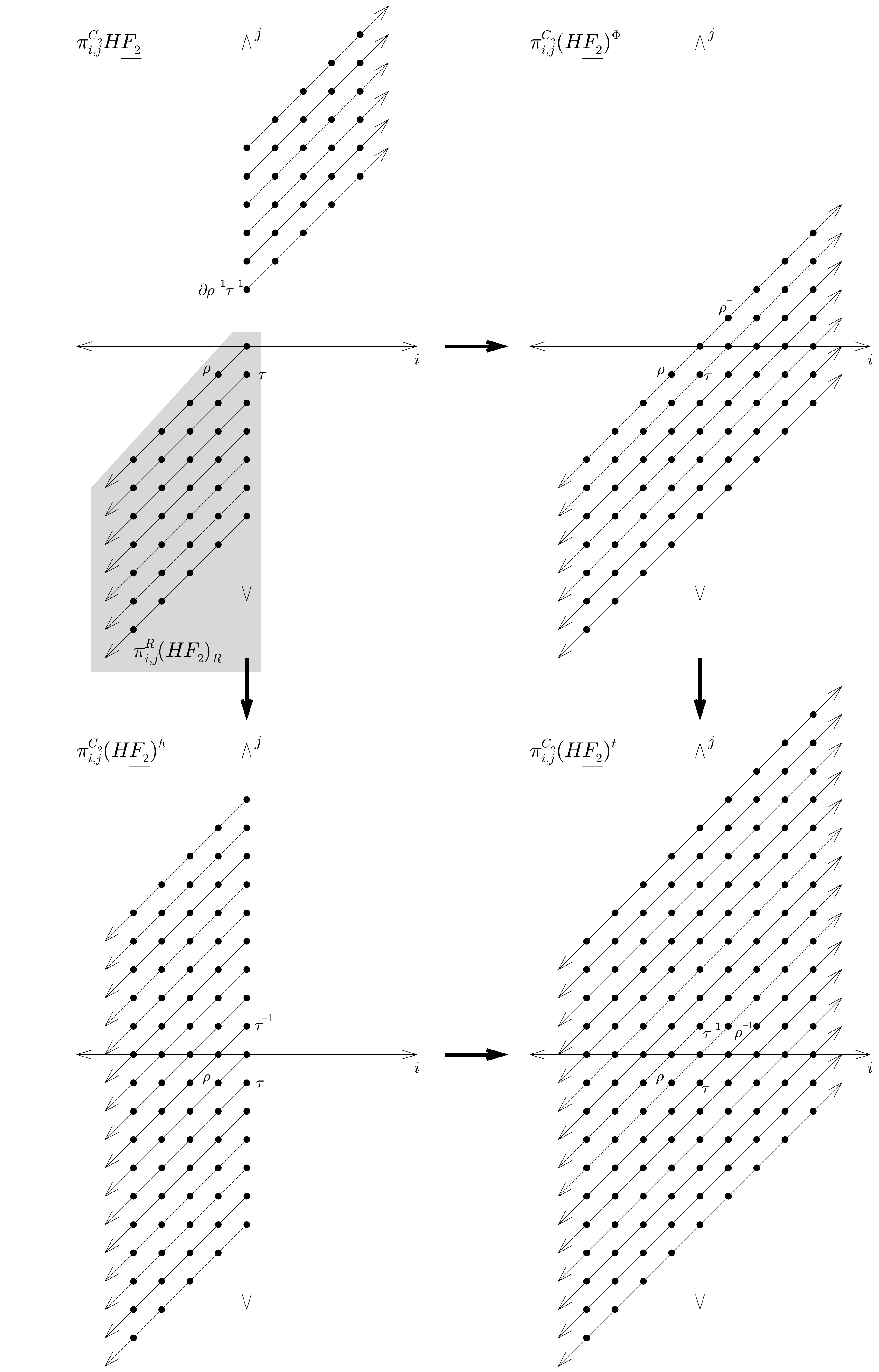}
\caption{Computing $\pi^{C_2}_{*,*}H\ul{\FF_2}$ from $\pi^{\RR}_{*,*}(H\FF_2)_\RR$.}
\label{fig:HF2}
\end{figure}

We now demonstrate the effectiveness of our theory by computing the $C_2$-equivariant homotopy groups of the $C_2$-Betti realizations of some $p$-complete cellular real motivic spectra from their motivic homotopy groups.  

For $p$ odd, the computational implementation of our theory is straightforward.  Given $X \in \SH_{\cell}(\RR)^\wedge_p$, we have (\ref{eq:rhosplitting})
$$ X \simeq X[\rho^{-1}] \vee X^\wedge_\rho $$
and we have
$$ \pi^{C_2}_{*,*}\widehat\Be{}^{C_2}_p(X) \simeq \pi^\RR_{*,*}X[\rho^{-1}] \oplus \pi^\RR_{*,*}X^{\wedge}_\rho[\tau^{-2}]. $$

In the case of $p = 2$, the computations are more interesting, and we illustrate this with some examples.
In each of these cases, the motivic homotopy groups are less complicated than the corresponding $C_2$-equivariant homotopy groups.\footnote{It is worth pointing out that in each of these examples the actual determination of these motivic homotopy groups is often the result of deep results in motivic homotopy theory, wheras the corresponding equivariant computations do not depend on similarly deep input.}

We point out that the use of isotropy separation to organize the equivariant homotopy of the examples in this section is not new --- see, for example, \cite{Greenlees}.

\subsection*{Mod $2$ motivic cohomology}

Let $(H\FF_2)_\RR \in \SH(\RR)$ denote the mod $2$ real motivic Eilenberg MacLane spectrum.  Dugger-Isaksen proved that the motivic complex cobordism spectrum $\mit{MGL}$ is cellular \cite{DuggerIsaksencell}.  Work of Hopkins-Morel and Hoyois \cite{Hoyois} implies that $(H\ZZ)_\RR$ (and hence $(H\FF_2)_\RR$) is a regular quotient of $\mit{MGL}$, and is therefore cellular.  Finally, Heller-Ormsby \cite[Thm.~4.17]{HellerOrmsby} prove that for any abelian group, the $C_2$-Betti realization of $(HA)_\RR$ is $H\ul{A}$, the $C_2$-equivariant Eilenberg-MacLane spectrum associated to the constant Mackey functor $\ul{A}$, so we have
$$ \Be^{C_2}(H\FF_2)_\RR \simeq H\ul{\FF_2}. $$
We may therefore apply our theory to compute $\pi_*^{C_2}H\ul{\FF_2}$.  

Recall again that we have \cite{Voevodsky1}
$$ \pi^\RR_{*,*} (H\FF_2)_\RR = \FF_2[\tau, \rho]. $$
Using Thm.~\ref{thm:Bachmann}, we have
\begin{align*}
 \pi^{C_2}_{*,*} H\ul{\FF_2}^\Phi & \cong  \pi^{\RR}_{*,*} (H\FF_2)_\RR[\rho^{-1}] \\
 & = \FF_2[\tau, \rho^\pm].
\end{align*}
Using Thm.~\ref{thm:main}, we have 
\begin{align*}
 \pi^{C_2}_{*,*} H\ul{\FF_2}^h & \cong  \pi^{\RR}_{*,*} (H\FF_2)_\RR[\tau^{-1}] \\
 & = \FF_2[\tau^{\pm}, \rho].
\end{align*}
Because the Tate spectrum is the geometric localization of the homotopy completion, we may apply Thm.~\ref{thm:Bachmann} to the above to deduce
\begin{align*}
 \pi^{C_2}_{*,*} H\ul{\FF_2}^t & \cong  \pi^{\RR}_{*,*} (H\FF_2)_\RR[\tau^{-1}][\rho^{-1}] \\
 & = \FF_2[\tau^{\pm}, \rho^\pm].
\end{align*}
We may then use the Mayer-Vietoris sequence
$$ \cdots \to 
\pi^{C_2}_{*+1,*} H\ul{\FF_2}^t \xto{\partial}  
\pi^{C_2}_{*,*} H\ul{\FF_2} \to
\pi^{C_2}_{*,*} H\ul{\FF_2}^h \oplus
\pi^{C_2}_{*,*} H\ul{\FF_2}^\Phi \to
\cdots
$$
associated to the isotropy separation square (\ref{eq:isotropy}) to deduce:
$$
\pi^{C_2}_{*,*}H\ul{\FF_2} = \FF_2[\tau, \rho] \oplus \frac{\FF_2[\tau, \rho]}{(\tau^\infty, \rho^{\infty})}\{ \partial \rho^{-1} \tau^{-1} \}.
$$

The calculation is displayed in Figure~\ref{fig:HF2}.  The motivic homotopy $\pi^{\RR}_{*,*}(H\FF_2)_\RR$ is displayed in the shaded region.  In this figure, a dot represents a factor of $\FF_2$, and a line represents multiplication by the element $\rho$.  The other three quadrants are then obtained from this motivic homotopy by inverting $\tau$, $\rho$, or both $\tau$ and $\rho$.  The resulting equivariant homotopy, deduced from the Mayer-Vietoris sequence, is displayed in the upper left hand chart (the combination of the shaded and unshaded regions).

\begin{figure}
\centering
\includegraphics[width=\linewidth]{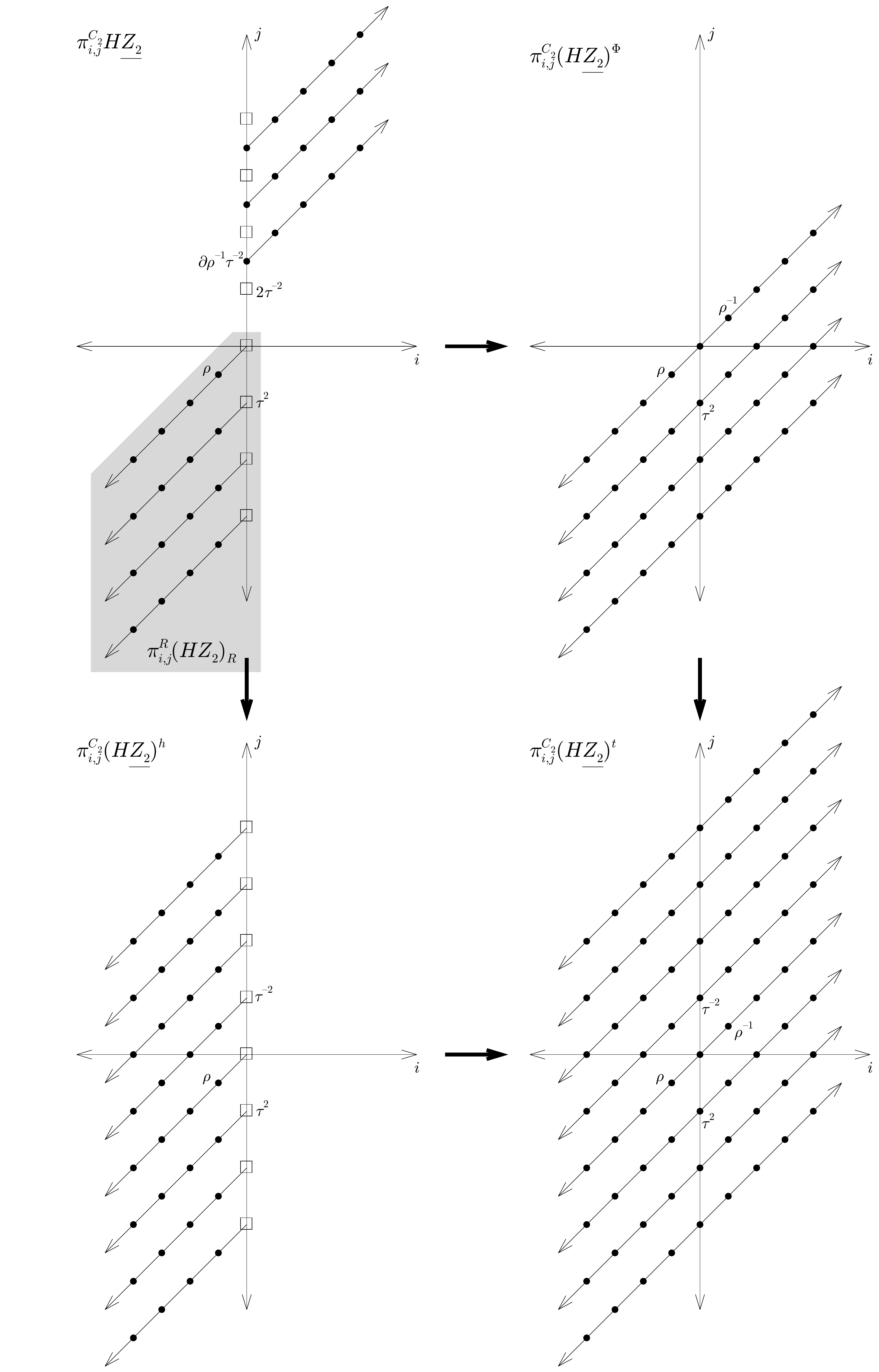}
\caption{Computing $\pi^{C_2}_{*,*}H\ul{\ZZ_2}$ from $\pi^{\RR}_{*,*}(H\ZZ_2)_\RR$.}
\label{fig:HZ}
\end{figure}

\subsection*{$2$-adic motivic cohomology}

The discussion of the previous subsection also establishes that the $2$-adic real motivic Eilenberg-MacLane spectrum $(H\ZZ_2)_\RR$ is cellular (and it is clearly $2$-complete).  The coefficients of $(H\ZZ_2)_\RR$ are given by (see, for example \cite{Hill}\footnote{Hill computes the homotopy of $BPGL\bra{0}^{\wedge}_2$, which, by the work of Hopkins-Morel and Hoyois \cite{Hoyois} is equivalent to $H\ZZ_2$.}):
$$\pi^\RR_{*,*}(H\ZZ_2)_\RR = \frac{\ZZ_2[\rho, \tau^2]}{(2\rho)}. $$
Again, \cite[Thm.~4.17]{HellerOrmsby} implies that
$$ \Be^{C_2}(H\ZZ_2)_\RR \simeq H\ul{\ZZ_2}. $$
We therefore deduce:
\begin{align*}
 \pi^{C_2}_{*,*} H\ul{\ZZ_2}^\Phi & \cong  \pi^{\RR}_{*,*} (H\ZZ_2)_\RR[\rho^{-1}] \\
 & = \FF_2[\tau^2, \rho^\pm],
\end{align*}
\begin{align*}
 \pi^{C_2}_{*,*} H\ul{\ZZ_2}^h & \cong  \pi^{\RR}_{*,*} (H\ZZ_2)_\RR[\tau^{-2}] \\
 & = \frac{\ZZ_2[\tau^{\pm 2}, \rho]}{(2\rho)},
\end{align*}
and
\begin{align*}
 \pi^{C_2}_{*,*} H\ul{\ZZ_2}^t & \cong  \pi^{\RR}_{*,*} (H\ZZ_2)_\RR[\tau^{-2}][\rho^{-1}] \\
 & = \FF_2[\tau^{\pm 2}, \rho^\pm].
\end{align*}
We therefore deduce from the Mayer-Vietoris sequence
$$ 
\pi^{C_2}_{*,*} H\ul{\ZZ_2} \cong \frac{\ZZ_2[\tau^{2}, \rho, 2\tau^{-2k}]}{(2\rho)} \oplus \frac{\FF_2[\tau^2, \rho]}{(\tau^\infty, \rho^\infty)}\{\partial \rho^{-1}\tau^{-2}\}.
$$
Note that there are implicitly defined relations in the above presentation, such as $\tau^2 (2\tau^{-2k}) = 2 \tau^{-2k+2}$ and $\rho (2\tau^{-2k}) = 0$.

The calculation is displayed in Figure~\ref{fig:HZ}.  Everything is analogous to the notation of Figure~\ref{fig:HF2}, except that there are now boxes in addition to solid dots, which represent factors of $\ZZ_2$.

\subsection*{The effective cover of $2$-adic algebraic K-theory}

\begin{figure}
\centering
\includegraphics[width=\linewidth]{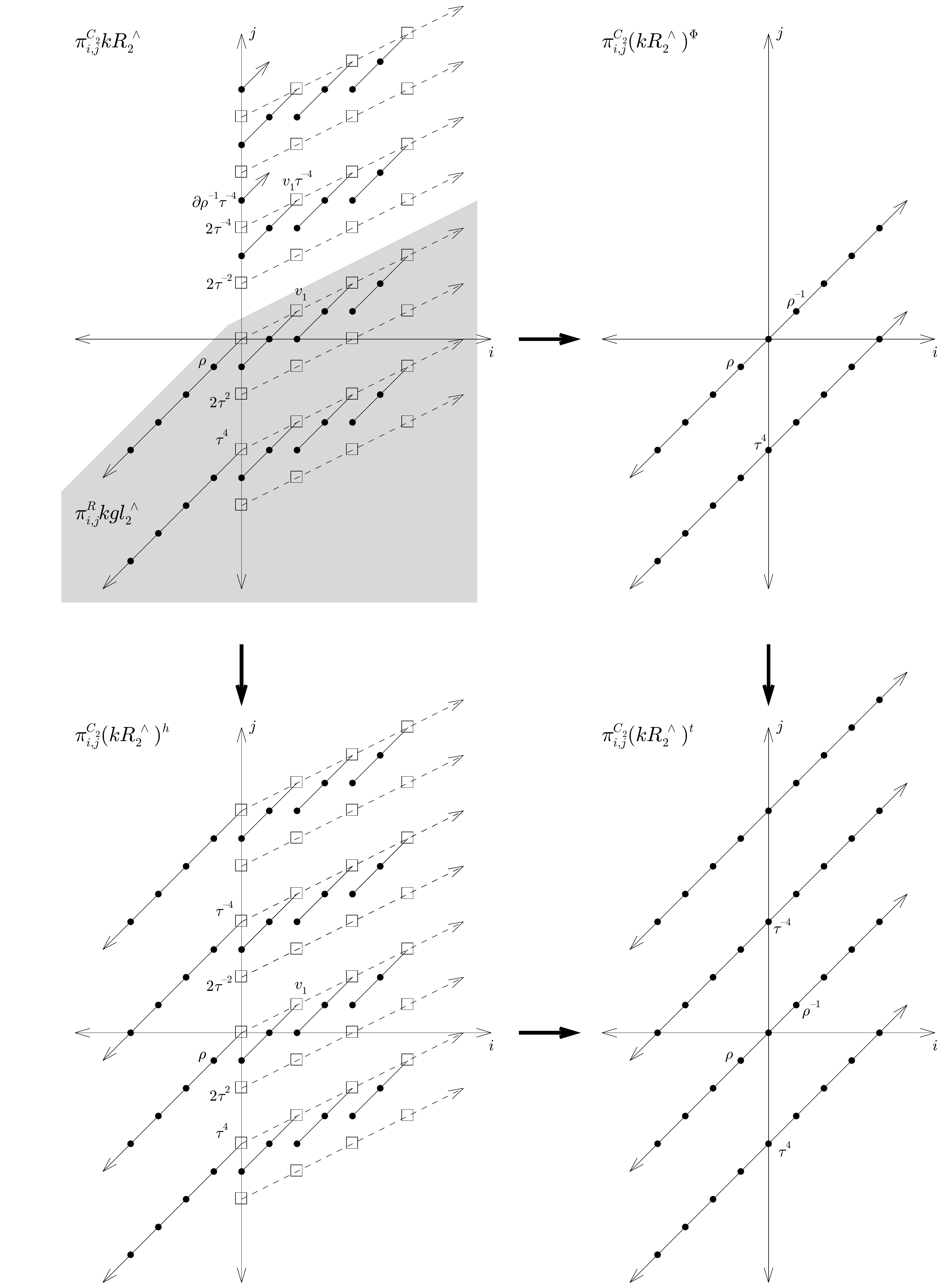}
\caption{Computing $\pi^{C_2}_{*,*}kR^{\wedge}_2$ from $\pi^{\RR}_{*,*}(kgl^{\wedge}_2)_\RR$.}
\label{fig:kgl}
\end{figure}

We now turn our attention to the spectrum $kgl$, the effective cover of $KGL$, the algebraic K-theory spectrum for the reals.  In \cite{Hill}, Hill computes the $2$-adic homotopy groups of this spectrum through the identification
$$ kgl^{\wedge}_2 \simeq BPGL\bra{1}^\wedge_2. $$
In particular, $kgl^{\wedge}_2$ is cellular.  We have
$$ \pi^\RR_{*,*}kgl^{\wedge}_2 \cong \frac{\ZZ_2[\rho, 2\tau^2, \tau^4, v_1]}{(2\rho, v_1\rho^3)} $$
with 
$$ \abs{v_1} = (2,1). $$
Note that, just as in the previous subsection, our presentation has implicitly defined relations, such as $(2\tau^2)^2 = 4\tau^4$.

It is clear from the definition of $KGL$ that we have
$$ \Be^{C_2} KGL = KR $$
where $KR$ is Atiyah's Real K-theory spectrum, and from \cite[Cor.~5.9]{Heard} we deduce the connective analog
$$ \Be^{C_2} kgl \simeq kR. $$
We deduce:
\begin{align*}
 \pi^{C_2}_{*,*} (kR^{\wedge}_2)^\Phi & \cong  \pi^{\RR}_{*,*} kgl^{\wedge}_2[\rho^{-1}] \\
 & = \FF_2[\tau^4, \rho^\pm],
\end{align*}
\begin{align*}
 \pi^{C_2}_{*,*} (kR^{\wedge}_2)^h & \cong  \pi^{\RR}_{*,*} kgl^{\wedge}_2[\tau^{-4}] \\
 & = \frac{\ZZ_2[\rho, 2\tau^2, \tau^{\pm 4}, v_1]}{(2\rho, v_1\rho^3)},
\end{align*}
and
\begin{align*}
 \pi^{C_2}_{*,*} (kR^{\wedge}_2)^t & \cong  \pi^{\RR}_{*,*} (kgl^{\wedge}_2)[\tau^{-4}][\rho^{-1}] \\
 & = \FF_2[\tau^{\pm 4}, \rho^\pm].
\end{align*}
We therefore deduce
$$ \pi^{C_2}_{*,*}kR^{\wedge}_2 \cong \frac{\ZZ_2[\rho, 2\tau^2, \tau^4, v_1, 2\tau^{-2k}, v_1\tau^{-4k}]}{(2\rho, v_1\rho^3)} \oplus \frac{\FF_2[\tau^4, \rho]}{(\tau^\infty, \rho^\infty)}\{\partial \rho^{-1}\tau^{-4}\}. $$

The calculation is displayed in Figure~\ref{fig:kgl}.  In this figure, dotted lines represent $v_1$-multiplication.

\bibliographystyle{amsalpha}

\nocite{*}
\bibliography{RC2}

\end{document}